\documentclass[a4paper,reqno]{amsart}
\usepackage{amssymb}
\usepackage{amsmath,amssymb,mathrsfs,latexsym,tikz,cite,hyperref}
\usepackage{amsfonts}
\usepackage{genyoungtabtikz}
\usetikzlibrary{tikzmark}
\usetikzlibrary{decorations.pathreplacing}
\usetikzlibrary{arrows,shapes,calc}
\usepackage{rotating, graphicx}
\usepackage{lscape,marginnote}
\usepackage{mathtools}
\usepackage{nicematrix}

\usepackage{amsaddr}
\usepackage{subcaption}
\usepackage{amssymb}
\usepackage[cmtip,all]{xy}

\usepackage{fancyhdr} 
\usepackage{indentfirst}
\usepackage{newlfont}
\usepackage{hyperref}
\usepackage{amsthm}
\usepackage{bm}
\usepackage{dsfont}
\usepackage{tikz-cd}
\tikzcdset{every label/.append style = {font = \small}}

\usepackage{ifthen}
\usepackage{xkeyval}
\usepackage{xcolor}
\usepackage{calc}

\usepackage[colorinlistoftodos]{todonotes} %package to make a list of all the to-dos

\title{Equivalence of $v$-decomposition matrices for blocks of Ariki-Koike algebras}
\author[A.~Dell'Arciprete]{alice Dell'Arciprete}
\address{School of Mathematics, \\University of East Anglia, \\Norwich NR4 7TJ, UK.}
\email{alice.dellarciprete@gmail.com}
\thanks{}

\numberwithin{equation}{section}
\numberwithin{figure}{section}
\newtheorem{thm}{Theorem}[section]
\newtheorem{prop}[thm]{Proposition}
\newtheorem{lem}[thm]{Lemma}
\newtheorem{cor}[thm]{Corollary}

\theoremstyle{definition}
\newtheorem{defin}[thm]{Definition}
\newtheorem{exe}[thm]{Example}

\theoremstyle{rem}
\newtheorem{rem}[thm]{Remark}

\newcommand{\Z}{\mathbb{Z}}
\newcommand{\la}{\lambda}

\definecolor{bead}{gray}{0.2}

\newcommand{\bd}{\begin{picture}(8,6)
\put(4,-1){\line(0,1){8}}
\put(4,3){\circle*{6}}
\end{picture}}
\newcommand{\nb}{\begin{picture}(8,6)
\put(4,-1){\line(0,1){8}}
\put(3,3){\line(1,0){2}}
\end{picture}}

\newcommand{\bigbd}{\begin{picture}(16,12)
\put(8,-2){\line(0,1){22}}
\put(8,6){\circle*{12}}
\end{picture}}
\newcommand{\bignb}{\begin{picture}(16,12)
\put(8,-2){\line(0,1){22}}
\put(6,6){\line(1,0){4}}
\end{picture}}

\newcommand{\bigvd}{\begin{picture}(16,20)
\put(8,10){\circle*{2}}
\put(8,4){\circle*{2}}
\put(8,16){\circle*{2}}
\end{picture}}

\begin{document}
\begin{abstract}
We consider the representation theory of the Ariki-Koike algebra, a $q$-deformation of the group algebra of the complex reflection group $C_r \wr \mathfrak{S}_n$. %The representations of this algebra are naturally indexed by multipartitions of $n$.
We examine blocks of the Ariki-Koike algebra. In particular, we prove a sufficient condition such that restriction of modules leads to a natural correspondence between the multipartitions of $n$ whose Specht modules belong to a block $B$ and those of $n-\delta_i(B)$ whose Specht modules belong to the block $B'$, obtained from $B$ applying a Scopes' equivalence. This bijection gives us an equivalence for the {$v$-decomposition numbers} of the Ariki-Koike algebras.
\end{abstract}
\maketitle
\section{Introduction}
Let $n$ be a positive integer. Let $\mathfrak{S}_n$ denote the symmetric group of degree $n$. This has the famous Coxeter presentation with generators $T_1,\ldots, T_{n-1}$ and relations
\begin{align*}
T_i^2&= 1, & \text{ for } 1 &  \leq i \leq n - 1,\\
T_i T_j &= T_j T_i, & \text{ for }1 &  \leq  i < j - 1 \leq n - 2,\\
T_i T_{i+1} T_i &= T_{i+1} T_i T_{i+1}, &  \text{ for } 1 &  \leq  i \leq  n-2.
\end{align*}
If we view this as a presentation for a (unital associative) algebra over a field $\mathbb{F}$, then the algebra we get is the group algebra $\mathbb{F}\mathfrak{S}_n$.\newline
Let $q$ be a non-zero element of $\mathbb{F}$. Now we can introduce a ‘deformation’, by replacing the relation $T_i^2=1$ with
$$(T_i + q)(T_i - 1) = 0$$ for each $i$. The resulting algebra is the Iwahori-Hecke algebra $H_n = H_{\mathbb{F},q}(\mathfrak{S}_n)$ of the symmetric group $\mathfrak{S}_n$. This algebra (of which the group algebra
$\mathbb{F}\mathfrak{S}_n$ is a special case) arises naturally, and its representation theory has been extensively studied. An excellent introduction to this theory is provided by Mathas's book \cite{Mat99}. As long as $q$ is non-zero, the representation theory of $H_n$ bears a remarkable resemblance to the representation theory of $\mathfrak{S}_n$. Indeed, there are many theorems concerning $H_n$ which reduce representation-theoretic notions to statements about the combinatorics of partitions.

In this paper, we consider the representation theory of the Ariki-Koike algebra. This algebra is a deformation of the group algebra of the complex reflection group $C_r \wr \mathfrak{S}_n$, defined using parameters $q, Q_1, \ldots, Q_r \in \mathbb{F}$. The representation theory of these algebras is beginning to be well understood. For example, the simple modules of the Ariki-Koike algebras have been classified; the blocks are known; and, in principle, the decomposition matrices of the Ariki-Koike algebras can be computed in characteristic zero. However, all known algorithms for computing decomposition matrices are recursive and in practice it is only possible to compute them for small values of $n$.  A comprehensive review of the representation theory of the Ariki-Koike algebras can be found in Mathas's paper \cite{Mat04}. In many respects it seems that the Ariki-Koike algebra behaves in the same way as the Iwahori-Hecke algebra $H_n$; many of the combinatorial theorems concerning $H_n$ have been generalised to the Ariki-Koike algebra, with the role of partitions being played by multipartitions. In fact, much of the difficulty of understanding the Ariki-Koike algebra seems to lie in finding the right generalisations of the combinatorics of partitions to multipartitions - very simple combinatorial notions (such as the definition of an $e$-restricted partition) can have rather nebulous generalisations (such as ‘Kleshchev’ multipartitions). 

In \cite{Scopes91}, under some conditions, Scopes establishes a natural correspondence between Specht modules and simple modules in the blocks $B$ and $\phi_i(B)$ of the symmetric groups where $\phi_i$ is the map swapping the runners $i-1$ and $i$ of the abacus display of each partition in the block $B$. This leads to an equivalence of decomposition matrices for these two blocks, meaning that the blocks $B$ and $\phi_i(B)$ have the same decomposition matrices.
In the last part of this paper Scopes proves Donovan's conjecture for blocks of the symmetric groups. In particular, given a partition $\lambda$ of $n$, the bijection $\phi_i$ gives a Morita equivalence between the blocks $B$ of $\lambda$ and the block $\phi_i(B)$ of $\phi_i(\lambda)$ for the symmetric group algebra $\mathbb{F}\mathfrak{S}_n$. The generalisation of Donovan's conjecture and hence of the Morita equivalence between the blocks $B$ and $\Phi_i(B)$ of Ariki-Koike algebras {(with $\Phi_i$ the map $\phi_i$ extended componentwise)} can be {seen as a special case of Theorem 3.3 in \cite{web23} where Webster proves this using $t$-exact Chuang-Rouquier equivalences in the more general setting of highest weight categorifications.} %An analougous result in the Ariki-Koike algebra case can be found in \cite{omarischaps} where Amara-Omari and Schaps combine some combinatorics with the Chuang-Rouquier categorification of integrable highest weight modules over Kac-Moody algebra of affine type $A$.
We would like to underline the fact that Morita equivalence does not imply the equivalence of decomposition matrices. Indeed, if we consider $n=8$ and $p=3$, we have that the block of the partition $(8)$ and the block of the partition $(1^{8})$ for the symmetric group algebra $\mathbb{F}_3\mathfrak{S}_8$ are Morita equivalent, but they have different decomposition matrices.

{This paper is intended as the (graded) Ariki-Koike algebra version of Scopes' paper about decomposition numbers.  In particular, we prove a sufficient condition such that two blocks of the Ariki-Koike algebras have the same graded decomposition matrices and thus they have the same decomposition matrices.}  In Section~\ref{basicdef}, we define the Ariki-Koike algebras and the combinatorial objects that we will use. Thus, we introduce the $r$-multipartition of an integer $n$ and we describe the construction of its abacus configuration with $e$ runners.  In Section \ref{Faybackground}, we summarise all the relevant results from \cite{Fay06} and \cite{Fay07} concerning weight of multipartitions and core blocks of Ariki-Koike algebras. In Section~\ref{multicoreresults}, we show that there is a sequence of multicores with non-increasing weights from a given multicore to a multicore in a core block. Then, we give a lower bound for the minimal difference between the positions of the lowest beads of two consecutive runners. In Section \ref{main_res}, we use this lower bound to get a condition on the weight of the block $B$ of a multipartition so that no configuration $ 
        \resizebox{.059\textwidth}{!}{$\begin{matrix}
i-1 & i \\
        \bd & \nb \\

        \end{matrix}$}
$ appears in the abacus display of any multipartition in $B$. In turn, we show that there is a natural correspondence between {graded Specht modules and graded simple modules} in the blocks $B$ and $\Phi_i(B)$ of the Ariki-Koike algebras.

\section{Basic definitions}\label{basicdef}

\subsection{The Ariki-Koike algebras}\label{AKdef}
Let $r\geq 1$ and $n \geq0$. Let $W_{r,n}$ be the complex reflection group $C_r \wr~\mathfrak{S}_n$. %This has a `Coxeter-like' presentation with generators $T_0, \ldots, T_{n-1}$ and relations
%\todo{\tiny not necessary?}
%\begin{align*}
%T_0^r&= 1,\\
%T_0T_1T_0T_1&= T_1T_0T_1T_0,\\
%T^2_i&= 1, & \text{for }1& \leq i \leq n - 1,\\
%T_iT_j&= T_jT_i, &\text{for }0& \leq i< j-1 \leq n - 2,\\
%T_iT_{i+1}T_i&= T_{i+1}T_iT_{i+1}, &\text{for }1& \leq  i \leq n - 2.
%\end{align*}
We can define the Ariki-Koike algebra as a deformation of the group algebra $\mathbb{F}W_{r,n}$.
\begin{defin}
Let $\mathbb{F}$ be a field and $q,Q_1,\ldots,Q_r$ be elements of $\mathbb{F}$, with $q$ non-zero. Let $\bm{Q}=(Q_1, \ldots, Q_r)$. The \textbf{Ariki-Koike algebra} $\mathcal{H}_{\mathbb{F},q,\bm{Q}}(W_{r,n})$ of $W_{r,n}$ is defined to be the unital associative $\mathbb{F}$-algebra with generators $T_0, \ldots, T_{n-1}$ and relations
\begin{align*}
(T_0-Q_1)\cdots(T_0-Q_r)&=0,\\
T_0T_1T_0T_1&= T_1T_0T_1T_0,\\
(T_i+1)(T_i-q)&=0, & \text{for }1 &\leq i \leq n - 1,\\
T_iT_j &= T_jT_i, &\text{for }0 &\leq i < j-1 \leq n - 2,\\
T_iT_{i+1}T_i &= T_{i+1}T_iT_{i+1}, &\text{for } 1&\leq  i \leq n - 2.\\
\end{align*}
\end{defin}
For brevity, we may write $\mathcal{H}_{r,n}$ for $\mathcal{H}_{\mathbb{F},q, \bm{Q}}(W_{r,n})$. Define $e$ to be minimal such that $1 + q + \ldots + q^{e-1}= 0$, or set $e = \infty$ if no such value exists. Throughout this paper we shall assume that $e$ is finite and we shall refer to $e$ as the \textbf{quantum characteristic}. Set $I = \{0, 1, \ldots, e-1\}$: we will identify $I$ with $\mathbb{Z}/e\mathbb{Z}$. We refer to any $r$-tuple of integers $(a_1, \ldots, a_r) \in \mathbb{Z}^r$ as a \textbf{multicharge}. We say $\bm{Q}$ is \textbf{$q$-connected} if, for each $i\in\{1, \ldots, r\}$, $Q_i = q^{a_i}$ for some $a_i \in \mathbb{Z}$. In \cite{dm02}, Dipper and Mathas prove that any Ariki-Koike algebra is Morita equivalent to a direct sum of tensor products of smaller Ariki-Koike algebras, each of which has $q$-connected parameters. Thus, we may assume that we are always working with a Ariki-Koike algebra with each $Q_i$ being an integral power of $q$. So we assume that we can find an $r$-tuple of integers $\bm{\mathrm{a}} = (a_1, \ldots, a_r)$ such that $Q_i = q^{a_i}$ for each $i$ where $q \neq 1$, so that $q$ is a primitive $e^{\text{th}}$ root of unity in $\mathbb{F}$. We call such $\bm{\mathrm{a}}$ a \textbf{multicharge} of $\mathcal{H}_{r,n}$.
%If $r = 1$ the cyclotomic relation collapses to $T_0 = Q_1 \in \mathbb{F}$ and we obtain the Hecke algebra of type $A$ which is independent of \bm\kappa. We shall write $\mathcal{H}_n$ for $\mathcal{H}_{1,n}$.
Since $e$ is finite then we may change any of the $a_i$ by adding a multiple of $e$, and we shall still have $Q_i = q^{a_i}$. %If $e = \infty$, then we have only one possible choice of multicharge $\kappa$.

For the rest of this section we will fix a multicharge $\bm{\mathrm{a}}=(a_1, \ldots, a_r)\in \mathbb{Z}^r$. Notice that at this stage we are not requiring that $\bm{\mathrm{a}}$ is a multicharge of the Ariki-Koike algebra.

\subsection{Multipartitions}\label{multpar}
A \textbf{partition} of $n$ is defined to be a non-increasing sequence $\lambda = (\lambda_1, \lambda_2, \dots)$ of non-negative integers whose sum is $n$. {The integers $\lambda_b$, for $b\geq1$, are called the \textbf{parts} of $\lambda$. We write $|\lambda| = n$.\\
Since $n < \infty$, there is a $k$ such that $\lambda_b = 0$ for $b > k$ and we write $\lambda = (\lambda_1, \dots , \lambda_k)$. We write $\varnothing$ for the unique empty partition $(0, 0, \dots)$. If a partition has repeated parts, for convenience we group them together with an index. For example, $$(4,4,2,1,0,0, \dots) = (4,4,2,1) = (4^2,2,1).$$
The \textbf{Young diagram} of a partition $\lambda$ is the subset
$$[\lambda] := \{(b,c) \in \mathbb{N}_{>0} \times \mathbb{N}_{>0
}\text{ } | \text{ } c\leq \lambda_b\}.$$
\begin{defin}\label{multipar}
An $r$-\textbf{multipartition} of $n$ is an ordered $r$-tuple $\bm{\lambda} = (\lambda^{(1)}, \dots, \lambda^{(r)})$  of partitions such that $$|\bm{\lambda}| := |\lambda^{(1)}|+ \ldots +|\lambda^{(r)}| = n.$$
If $r$ is understood, we shall just call this a multipartition of $n$.
\end{defin}
As with partitions, we write the unique multipartition of $0$ as $\bm\varnothing$. The \textbf{Young diagram} of a multipartition $\bm\lambda$ is the subset
$$
[\bm{\lambda}]:=\{(b,c,j)\in \mathbb{N}_{>0}\times \mathbb{N}_{>0}\times\{1, \ldots, r\} \text{ }|\text{ } c \leq \lambda_b^{(j)}\}.
$$
We may abuse notation by not distinguishing a multipartition from its Young diagram.}
The elements of $[\bm{\lambda}]$ are called \textbf{nodes} of $\bm\lambda$. We say that a node $\mathfrak{n} \in [\bm{\lambda}]$ is \textbf{removable} if $[\bm{\lambda}]\setminus \{\mathfrak{n}\}$ is also the Young diagram of a multipartition. We say that an element $\mathfrak{n} \in \mathbb{N}^2_{>0}\times \{1, \ldots, r\}$ is an \textbf{addable node} if $\mathfrak{n} \notin [\bm{\lambda}]$ and $[\bm{\lambda}] \cup \{\mathfrak{n}\}$ is the Young diagram of a multipartition. {Let $e \in \{2, 3, \ldots\}$. Given a multicharge $\bm{\mathrm{a}}=(a_1, \ldots, a_r)$, to each node $(b,c,j) \in [\bm\lambda]$ we associate its \textbf{residue} $\mathrm{res}_{\bm{\mathrm{a}}}(b,c,j) = a_j + c-b$ $(\text{mod }e)$. We draw the residue diagram of $\bm\lambda$ by replacing each node in the Young diagram by its residue.}
\begin{exe}\label{exsameblock}
Suppose $r=3$ and $\bm{\mathrm{a}}=(1,0,2)$. Let $\bm{\lambda}=((1^2),(2),(2,1))$ and $\bm\mu=((1),(2,1),(1^3))$ be two multipartitions of $7$. If $e=4$, then the $4$-residue diagram of $[\bm{\lambda}]$ and of $[\bm\mu]$ are
$$
\young(1,0) \quad \young(01) \quad \young(<2><3>,1) \quad \text{ and } \quad \young(1) \quad \young(01,3) \quad \young(2,1,0).
$$
\end{exe}
For an $r$-multipartition $\bm\lambda$, define the \textbf{residue set} of $\bm\lambda$ to be the multiset $\mathrm{Res}_{\bm{\mathrm{a}}}(\bm\lambda) = \{\mathrm{res}(\mathfrak{n}) \text{ } | \text{ } \mathfrak{n} \in [\bm\lambda]\}$. Notice that in the example above
\begin{equation*}
\mathrm{Res}_{\bm{\mathrm{a}}}(\bm\lambda)= \{0,0,1,1,1,2,3\} = \mathrm{Res}_{\bm{\mathrm{a}}}(\bm\mu).
\end{equation*}

\subsection{Kleshchev multipartitions}
Residues of nodes are useful in classifying the simple $\mathcal{H}_{r,n}$-modules. Indeed, the notion of residue helps us to describe a certain subset $\mathcal{K}$ of the set of all multipartitions, which index the simple modules for $\mathcal{H}_{r,n}$. The multipartitions in the set $\mathcal{K}$ are called \textbf{Kleshchev multipartitions} and they are a subset of $e$-restricted multipartitions (i.e., the set of multipartition with all $e$-restricted partitions as components). The recursive definition of Kleshchev multipartition can be found in \cite{Fay07}.  For the scope of this paper, it is enough to consider them as the multipartitions indexing the simple modules of $\mathcal{H}_{r,n}$.

\subsection{Specht module and simple modules}\label{specht}
The algebra $\mathcal{H}_{r,n}$ is a cellular algebra \cite{djm,gl} with the cell modules indexed by $r$-multipartitions of $n$. For each $r$-multipartition $\bm\lambda$ of $n$ we define a $\mathcal{H}_{r,n}$-module $S^{\bm\lambda}$ called a \textbf{Specht module}; these modules are the cell modules given by the cellular basis of $\mathcal{H}_{r,n}$ \cite{djm}. When $\mathcal{H}_{r,n}$ is semisimple, the Specht modules form a complete set of non-isomorphic irreducible $\mathcal{H}_{r,n}$-modules. However, we are mainly interested in the case when $\mathcal{H}_{r,n}$ is not semisimple. In this case, for each Kleshchev $r$-multipartition $\bm\lambda$, the Specht module $S^{\bm\lambda}$ has an irreducible cosocle $D^{\bm\lambda}$, and the $D^{\bm\lambda}$ provide a complete set of irreducible modules for $\mathcal{H}_{r,n}$ as $\bm\lambda$ ranges over the set of Kleshchev multipartitions of $n$ \cite[Theorem 4.2]{ari01}. If $\bm\lambda$ and $\bm\mu$ are $r$-multipartition of $n$ with $\bm\mu$ Kleshchev, let $d_{\bm\lambda \bm\mu} = [S^{\bm\lambda}:D^{\bm\mu}]$ denote the multiplicity of the simple module $D^{\bm\mu}$ as a composition factor of the Specht module $S^{\bm\lambda}$. The matrix $D = (d_{\bm\lambda \bm\mu})$ is called the \textbf{decomposition matrix} of $\mathcal{H}_{r,n}$ and determining its entries is one of the most important open problems in the representation theory of the Ariki-Koike algebras. It follows from the cellularity of $\mathcal{H}_{r,n}$ that the decomposition matrix is `triangular'; to state this we need to define the dominance order on multipartitions. Given two multipartitions $\bm{\lambda}$ and $\bm{\mu}$ of $n$, we say that $\bm{\lambda}$ \textbf{dominates} $\bm{\mu}$, and write $\bm{\lambda} \trianglerighteq \bm{\mu}$, if {
$$\sum_{a=1}^{j-1}|\lambda^{(a)}| + \sum_{b=1}^{i} \lambda_b^{(j)} \geq \sum_{a=1}^{j-1}|\mu^{(a)}| + \sum_{b=1}^{i}\mu_b^{(j)}$$
for $j=1,2, \ldots, r$ and for all $i \geq 1$}. Then we have the following.
\begin{thm}\cite{djm,gl}\label{decmtxHQ}
Suppose $\bm\lambda$ and $\bm\mu$ are $r$-multipartitions of $n$ with $\bm\mu$ Kleshchev.
\begin{enumerate}
\item If $\bm\mu=\bm\lambda$, then $[S^{\bm{\lambda}}:D^{\bm{\mu}}]=1$.
\item If $[S^{\bm{\lambda}}:D^{\bm{\mu}}]>0$, then $\bm{\lambda} \trianglerighteq \bm{\mu}$.
\end{enumerate}
\end{thm}

Note that the dominance order is a partial order on the set of multipartitions. The dominance order is certainly the `correct' order to use for multipartitions, but it is sometimes useful to have a total order, $>$, on the set of multipartitions. The one we use is given as follows. 
\begin{defin}
Given two multipartitions $\bm\lambda$ and $\bm\mu$ of $n$ with $\bm\lambda \neq \bm\mu$, we write $\bm\lambda > \bm\mu$ if and only if the minimal $j \in \{1, \ldots, r\}$ for which $\lambda^{(j)} \neq \mu^{(j)}$ and the minimal $i\geq 1$ such that $\lambda^{(j)}_i \neq \mu^{(j)}_i$ satisfy $\lambda^{(j)}_i > \mu^{(j)}_i$. This is called the \textbf{lexicographic order} on multipartitions.
\end{defin}

{\subsection{Graded decomposition matrices}\label{subsec:grdec}

Recent work has given us a new line of attack. The cyclotomic quiver Hecke algebras of type $A$, known as KLR algebras (defined independently by Khovanov and Lauda and by Rouquier \cite{KL09,Rou08}), have been shown to be isomorphic to Ariki-Koike algebras by Brundan and Kleshchev in \cite{bk}. Via this isomorphism, the $\Z$-grading of the KLR algebras can be used in the setting of Ariki-Koike algebras, and thus graded Specht modules \cite{bkw11} and graded decomposition numbers can be studied.

For a graded $\mathcal{H}_{r,n}$-module $D$, let $D\langle d \rangle$ denote the graded shift (by $d$) of the module $D$ -- in other words $D\langle d \rangle_s = D_{s-d}$.
Let $\bm\la$ and $\bm\mu$ be $r$-multipartitions of $n$ with $\bm\mu$ a Kleshchev multipartition. Then the corresponding \textbf{graded decomposition number} is the Laurent polynomial
$$
d_{\bm\la\bm\mu}(v) = [S^{\bm\la}: D^{\bm\mu}]_v = \sum_{d\in\mathbb{Z}} [S^{\bm\la}: D^{\bm\mu} \langle d \rangle] v^d \in \mathbb{N}[v,v^{-1}].
$$

It is known that $d_{\bm\la\bm\la}(v) = 1$ and $d_{\bm\la\bm\mu}(v) \neq 0$ only if $\bm\la \trianglerighteq \bm\mu$ -- this follows because we can recover the ordinary decomposition numbers by setting $v=1$.

Moreover, we can recover the ordinary decomposition numbers $d_{\bm\la\bm\mu}$ from $d_{\bm\la\bm\mu}(v)$ by setting $v$ to $1$.}

\subsection{Blocks of Ariki-Koike algebras}
%The blocks of the Iwahori-Hecke algebra are classified by the so-called `Nakayama conjecture'.
It follows from the cellularity of $\mathcal{H}_{r,n}$ that each Specht module $S^{\bm\lambda}$ lies in one block of $\mathcal{H}_{r,n}$, and we abuse notation by saying that a multipartition $\bm\lambda$ lies in a block $B$ if $S^{\bm\lambda}$ lies in $B$. On the other hand, each block contains at least one Specht module, so in order to classify the blocks of $\mathcal{H}_{r,n}$, it suffices to describe the corresponding partition of the set of multipartitions. We have the following classification of the blocks of $\mathcal{H}_{r,n}$.

\begin{thm}\cite[Theorem 2.11]{lm07}\label{block_car}
Let $\bm{\mathrm{a}} \in \Z^r$ be a multicharge of $\mathcal{H}_{r,n}$. Suppose $\bm\lambda$ and $\bm\mu$ are $r$-multipartitions of $n$. Then, $S^{\bm\lambda}$ and $S^{\bm\mu}$ lie in the same block of $\mathcal{H}_{r,n}$ if and only if $\mathrm{Res}_{\bm{\mathrm{a}}}(\bm\lambda)=\mathrm{Res}_{\bm{\mathrm{a}}}(\bm\mu)$.
\end{thm}

\begin{exe}
Continuing Example \ref{exsameblock}, we see that the residue sets of $\bm\lambda$ and $\bm\mu$ are equal. {Hence $\bm\lambda$ and $\bm\mu$ lie in the same block of $\mathcal{H}_{3,7}(q,\bm Q)$ where $e=4$ and $\bm Q = (q^1,q^0, q^2) = (q, 1, q^2).$}
\end{exe}

\subsection{Induction and restriction}
If $n>1$, then $\mathcal{H}_{r,n-1}$ is naturally a subalgebra of $\mathcal{H}_{r,n}$, and in fact $\mathcal{H}_{r,n}$ is free as an $\mathcal{H}_{r,n-1}$-module. So there are well-behaved induction and restriction functors between the module categories of $\mathcal{H}_{r,n-1}$ and $\mathcal{H}_{r,n}$. Given modules $M$, $N$ for $\mathcal{H}_{r,n-1}$ and $\mathcal{H}_{r,n}$, respectively, we write $M\uparrow^{\mathcal{H}_{r,n}}$ and $N\downarrow_{\mathcal{H}_{r,n-1}}$ for the induced and restricted modules. If $B$ and $C$ are blocks of $\mathcal{H}_{r,n-1}$ and $\mathcal{H}_{r,n}$, respectively, then we may write $M\uparrow^{C}$ and $N\downarrow_{B}$ for the projections of the induced and restricted modules onto $B$ and $C$.

%From the definition of the Specht modules, the action of $\mathcal{H}_{r,n-1}$ on $S^{\bm \lambda}\downarrow_{\mathcal{H}_{r,n-1}}$, for $\bm \lambda$ multipartition of $n$, is given by ignoring the node in the tableaux with label $n$. With only a small amount of work this implies the following result.

{Once we have a grading for the Specht module, it is natural to consider the graded branching rule for them. In order to state this result we need to introduce the following notation. We impose a partial order $>$ on the set of nodes of residue $i \in I$ of a multipartition by saying that $(b,c,j)$ is \textbf{above} $(b',c',j')$ (or $(b',c',j')$ is \textbf{below} $(b,c,j)$) if either $j<j'$ or ($j=j'$ and $b<b'$). In this case we write $(b,c,j)>(b',c',j')$. %Note this order restricts to a total order on the set of all addable and removable nodes of residue $i \in I$ of a multipartition.

Let $\bm\la$ be an $r$-multipartition of $n$, $i\in I$, $A$ be a removable $i$-node of 
$\bm\la$ and $B$ be an addable $i$-node of $\bm\la$. We set
\begin{equation}\label{EDMUA}
\begin{split}
N_{A}(\bm\la):= &\#\{\text{addable $i$-nodes of $\bm\la$ below $A$}\}
\\
&-\#\{\text{removable $i$-nodes of $\bm\la$ below  $A$}\};
\end{split}
\end{equation}
\begin{equation}\label{EDMUB}
\begin{split}
N^B(\bm\la):=&\#\{\text{addable $i$-nodes of $\bm\la$ above $B$}\}\\
&-\#\{\text{removable $i$-nodes of $\bm\la$ above $B$}\}.
\end{split}
\end{equation} 

\begin{thm}\cite[Main Theorem]{hm11}\cite[Theorem 4.11]{bkw11}\label{gradedbranching}
\begin{itemize}
\item Suppose $\bm \lambda$ is a multipartition of $n-1$, and let $B_1 > \ldots > B_s$ be all the addable nodes of $[\bm \lambda]$. For each $i = 1, \ldots, s$, let $\bm\lambda^{B_i}$ be the multipartition of $n$ with $[\bm\lambda^{B_i}] = [\bm\lambda] \cup \{B_i\}$. Then $S^{\bm\lambda}\uparrow^{\mathcal{H}_{r,n}}$ has a filtration
$$S^{\bm\lambda}\uparrow^{\mathcal{H}_{r,n}}=V_s \supset \dots \supset V_1 \supset V_0 = \{0\}$$
as a graded $\mathcal{H}_{r,n}$-module in which the factors are 
$$ V_i / V_{i-1} = S^{\bm\lambda^{B_i}}\langle N^{{B}_i}(\bm\la)\rangle$$
for all $1 \leq i \leq s$.

\item Suppose $\bm\lambda$ is a multipartition of $n$, and let $A_1 < \ldots < A_t$ be all the removable nodes of $[\bm\lambda]$. For each $i = 1, \ldots, t$, let $\bm\lambda_{A_i}$ be the multipartition of $n-1$ with $[\bm\lambda_{A_i}] = [\bm\lambda] \setminus \{A_i\}$. Then $S^{\bm\lambda}\downarrow_{\mathcal{H}_{r,n-1}}$ has a filtration
$$S^{\bm\lambda}\downarrow_{\mathcal{H}_{r,n-1}}=V_t \supset \dots \supset V_1 \supset V_0 = \{0\}$$
as a graded $\mathcal{H}_{r,n}$-module in which the factors are 
$$ V_i / V_{i-1} = S^{\bm\lambda_{A_i}}\langle N_{{A}_i}(\bm\la)\rangle$$
for all $1 \leq i \leq t$.
\end{itemize}
\end{thm}
If a module $M$ has a filtration $M=V_t \supseteq V_{t-1} \supseteq \ldots \supseteq V_{1} \supseteq V_0 = 0$ with $V_i/V_{i-1} \cong M_i\langle d_i\rangle$ for $i = 1, \ldots, t$, and $d_i\in \Z$, we write $$M \sim \sum_{i=1}^{t} M_i \langle d_i\rangle \text{ or } M \sim\sum_{i=1}^{t}v^{d_i} M_i, \qquad \text{for }d_i\in \Z.$$ }

%\begin{thm} \cite[Lemma 2.1]{ari}; \cite[Corollary 3.7]{Mat09}\label{Arikithm}
%\begin{itemize}
%\item Suppose $\bm \lambda$ is a multipartition of $n-1$, and let $\mathfrak{n}_1, \ldots, \mathfrak{n}_s$ be the addable nodes of $[\bm \lambda]$. For each $i = 1, \ldots, s$, let $\bm\lambda^{+i}$ be the multipartition of $n$ with $[\bm\lambda^{+i}] = [\bm\lambda] \cup \{\mathfrak{n}_i\}$. Then $S^{\bm\lambda}\uparrow^{\mathcal{H}_{r,n}}$ has a filtration in which the factors are $S^{\bm\lambda^{+1}}, \ldots, S^{\bm\lambda^{+s}}$.
%\item Suppose $\bm\lambda$ is a multipartition of $n$, and let $\mathfrak{n}_1, \ldots, \mathfrak{n}_t$ be the removable nodes of $[\bm\lambda]$. For each $i = 1, \ldots, t$, let $\bm\lambda^{-i}$ be the multipartition of $n-1$ with $[\bm\lambda^{-i}] = [\bm\lambda] \setminus \{\mathfrak{n}_i\}$. Then $S^{\bm\lambda}\downarrow_{\mathcal{H}_{r,n-1}}$ has a filtration in which the factors are $S^{\bm\lambda^{-1}}, \ldots, S^{\bm\lambda^{-t}}$.
%\end{itemize}
%\end{thm}
%If a module $M$ has a filtration $M=M_0 \supseteq M_{1} \supseteq \ldots \supseteq M_{t-1} \supseteq M_t = 0$ with $M_i/M_{i+1} \cong N_i$ for $i = 0, \ldots, t - 1$, we will write $$M \sim \sum_{i=0}^{t-1} N_i.$$     

\subsection{$\beta$-numbers and the abacus}\label{betaAbacus}
Given the assumption that the cyclotomic parameters of $\mathcal{H}_{r,n}$ are all powers of $q$, we may conveniently represent multipartitions on an abacus display. Fix $\bm{\mathrm{a}}=(a_1, \ldots, a_r)\in \mathbb{Z}^{r}$
to be a multicharge of $\mathcal{H}_{r,n}$.
\begin{defin}
Let $\bm{\lambda}=(\lambda^{(1)}, \ldots, \lambda^{(r)})$ be a multipartition of $n$. For every $i\geq1$ and for every $j \in \{1, \ldots, r\}$, we define the \textbf{$\beta$-number} $\beta_i^j$ to be
$$\beta_i^j := \lambda_i^{(j)} + a_j - i.$$
\end{defin}
The set $B_{a_j}^j =\{ \beta_1^j, \beta_2^j, \ldots \}$ is the set of $\beta$-numbers (defined using the integer $a_j$) of partition $\lambda^{(j)}$. It is easy to see that any set $B_{a_j}^j =\{ \beta_1^j, \beta_2^j, \ldots \}$ is a set containing exactly $a_j + N$ integers greater than or equal to $-N$, for sufficiently large $N$. 

For each set $B_{a_j}^j$ we can define an abacus display in the following way. We take an abacus with $e$ infinite vertical runners, which we label $0, 1, \ldots, e-1$ from left to right, %(or $\ldots,-1, 0, 1, \ldots$ from left to right, if $e = \infty$)
and we mark positions on runner $l$ and label them with the integers congruent to $l$ modulo $e$, so that %(if $e < \infty$) then 
position $(x+1)e+l$ lies immediately below position $xe + l$, for each $x$. Now we place a bead at position $\beta_i^j$, for each $i$. The resulting configuration is called the $e$-abacus display, or the $e$-abacus configuration, for $\lambda^{(j)}$ with respect to $a_j$. Moreover, we say that the bead corresponding to the $\beta$-number $xe+l$ is at \textbf{level} $x$ for $x \in \mathbb{Z}$. Then, we define $\ell^{\bm{\mathrm{a}}}_{ij}(\bm{\lambda})$ to be the level of the lowest bead on runner $i$ of the abacus display for $\lambda^{(j)}$ with respect to $\bm{\mathrm{a}}$. Note that this bead is the largest element of $B_{a_j}^j$ congruent to $i$ modulo $e$.

Hence, we can now define the $e$-\textbf{abacus display}, or the $e$-\textbf{abacus configuration}, for a multipartition $\bm{\lambda}$ with respect to $\bm{\mathrm{a}}$ to be the $r$-tuple of $e$-abacus displays associated to each component $\lambda^{(j)}$. If it is clear which $e$ we are referring to, we simply say abacus configuration. When we draw abacus configurations we will draw only a finite part of the runners and we will assume that above this point the runners are full of beads and below this point there are no beads. %In our examples, we will draw a line between the strictly negative and the positive positions. From now on, we will identify $\beta$-sets with abacus configurations.

\begin{exe}
Suppose that $r = 3$, $\bm{\mathrm{a}} = (-1, 0, 1)$ and $\bm{\lambda}= ((1), \varnothing, (1^2))$. Then we have
\begin{align*}
B^1_{-1} &= \{-1,-3,-4,-5, \ldots\};\\
B^2_0 &= \{-1,-2,-3,-4\ldots\};\\
B^3_1 &= \{1,0,-2,-3,-4,\ldots\}.
\end{align*}
So, the $e$-abacus display with respect to the multicharge $\bm{\mathrm{a}}$ for $\bm{\lambda}$ when $e = 4$ is
%\todo{find a way to draw the horizontal line in the abacus}
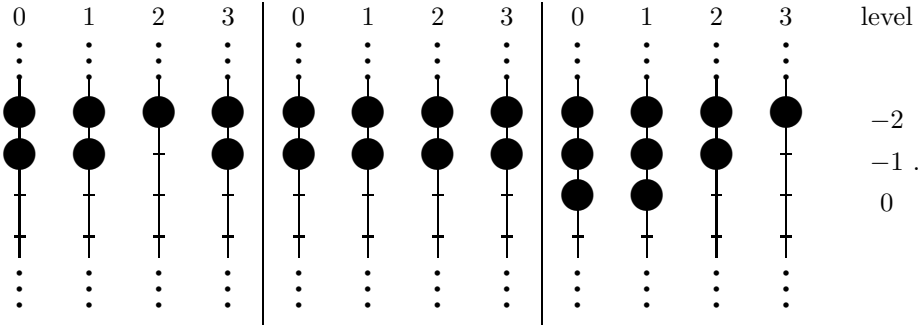
\begin{center}
\begin{tabular}{c|c|c}
$      
        \begin{matrix}

        0 & 1 & 2 & 3 \\

        \bigvd & \bigvd & \bigvd & \bigvd \\

        \bigbd & \bigbd & \bigbd & \bigbd \\
        
        \bigbd & \bigbd & \bignb & \bigbd \\
     
        \bignb & \bignb & \bignb & \bignb \\
        
        \bignb & \bignb & \bignb & \bignb \\

        \bigvd & \bigvd & \bigvd & \bigvd \\

        \end{matrix}
$
&
$     \begin{matrix}

        0 & 1 & 2 & 3 \\

        \bigvd & \bigvd & \bigvd & \bigvd \\

        \bigbd & \bigbd & \bigbd & \bigbd \\

        \bigbd & \bigbd & \bigbd & \bigbd \\

        \bignb & \bignb & \bignb & \bignb \\
        
        \bignb & \bignb & \bignb & \bignb \\

        \bigvd & \bigvd & \bigvd & \bigvd \\

        \end{matrix}
$
&
$     \begin{matrix}

        0 & 1 & 2 & 3 & & \text{level}\\

        \bigvd & \bigvd & \bigvd & \bigvd & &\\

        \bigbd & \bigbd & \bigbd & \bigbd & & -2\\

        \bigbd & \bigbd & \bigbd & \bignb & & -1\\

        \bigbd & \bigbd & \bignb & \bignb & & 0\\
        
        \bignb & \bignb & \bignb & \bignb & & \\

        \bigvd & \bigvd & \bigvd & \bigvd & & \\

        \end{matrix}
.$\\
\end{tabular}
\end{center}
\end{exe}

\subsection{Rim $e$-hooks and $e$-core}
We recall some other definitions about the diagram of a partition since their generalization for multipartitions will be really useful in this paper. 

\begin{defin}
Suppose $\lambda$ is a partition of $n$ and $(i,j)$ is a node of $[\lambda]$.
\begin{enumerate}
\item The \textbf{rim} of $[\lambda]$ is defined to be the set of nodes
$$\{(i,j) \in [\lambda] \text{ }|\text{ }(i+1,j+1) \notin [\lambda]\}.$$
\item Define an \textbf{$e$-rim hook} to be a connected subset $R$ of the rim containing exactly $e$ nodes such that $[\lambda]  \setminus R$ is the diagram of a partition.
\item  If $\lambda$ has no $e$-rim hooks then we say that $\lambda$ is an \textbf{$e$-core}.
\item If we can remove $w$ $e$-rim hooks from $[\lambda]$ to produce an $e$-core, then we say that $\lambda$ has \textbf{$e$-weight} $w$. In particular, an $e$-core has weight 0.
\end{enumerate}
\end{defin}

%\begin{exe}
%Let $e=4$. Consider the partition $\lambda=(3,2)$. Then the residue diagram is
%$$\young(0!\ylw12,20).$$
%We can remove a $4$-rim hook, that is the shaded one in the diagram. So, $\lambda$ has $4$-weight 1 and its $4$-core is $(1)$.
%\end{exe}

An abacus display for a partition is useful for visualising the removal of $e$-rim hooks. If we are given an abacus display for $\lambda$ with $\beta$-numbers in a set $B$, then $[\lambda]$ has a $e$-rim hook if and only if there is a $\beta$-number $\beta_i \in B$ such that $\beta_i-e \notin B$. Furthermore, removing a $e$-rim hook corresponds to reducing such a $\beta$-number by $e$. On the abacus, this corresponds to sliding a bead up one position on its runner. So, $\lambda$ is an $e$-core if and only if every bead in the abacus display has a bead immediately above it. Using this, we can see that the definition of $e$-weight and $e$-core of $\lambda$ are well defined.
%\begin{exe}
%Let $e=4$. Consider the partition $\lambda=(3,2)$ and $a=2$. Hence, looking at the following abacus display
%$$\begin{matrix}
%           0 & 1 & 2 & 3 \\
%        \bigvd & \bigvd & \bigvd & \bigvd \\
%        \bigbd & \bigbd & \bigbd & \bigbd \\
%        \bigbd & \bigbd & \bigbd & \bigbd \\
%        \bignb & \bignb & \bigbd & \bignb \\
%        \bigbd & \bignb & \bignb & \bignb \\        
%        \bignb & \bignb & \bignb & \bignb \\
%        \bigvd & \bigvd & \bigvd & \bigvd \\
%
%        \end{matrix},$$
%one can see that the only bead that can be slided up is the last one in runner 0. So, we can deduce that the partition $\lambda$ has $4$-weight 1 and its $4$-core  is the partition corresponding to the abacus configuration with all the beads as high as possible that is (1).
%\end{exe}
{Finally, we can also notice that each bead corresponds to the end of a row of the diagram of $\lambda$ (or to a row of length $0$).

For multipartitions, by the definition of the $\beta$-numbers the node at the end of the row (if it exists) has residue $i$ if and only if the corresponding bead is on runner $i$ of the abacus for $i \in I$.
Thus, if we increase any $\beta$-number by one, this is equivalent to moving a bead from runner $i$ to runner $i + 1$ which is equivalent to adding a node of residue $i + 1$ to the diagram of a component $\lambda^{(j)}$. Similarly, decreasing a $\beta$-number by one is equivalent to moving a bead from runner $i$ to runner $i - 1$ which is equivalent to removing a node of residue $i$ from the diagram of $\lambda^{(j)}$.}

\section{Background results about weight of multipartitions and core block}\label{Faybackground}
Here we summarise the main definitions and some results from \cite{Fay06} and \cite{Fay07}, mostly concerning weight of multipartitions and core blocks of Ariki-Koike algebras. These are the generalisations to multipartitions of the notions of $e$-weight and $e$-core of partitions. 

\subsection{Weight and hub of multipartitions}
In \cite{Fay06}, Fayers generalises the notion of weight from partitions to multipartitions which will be of great utility in this paper. Fix a multicharge $\bm{\mathrm{a}}=(a_1, \ldots, a_r)$ of $\mathcal{H}_{r,n}$.

\begin{defin}\label{mult_weight}
Let  $\bm{\lambda}= (\lambda^{(1)}, \ldots, \lambda^{(r)})$ be a multipartition of $n$. Let $c_i(\bm{\lambda})$ denote the number of nodes in $[\bm{\lambda}]$ of residue $i \in I$ and $\overline{a_j}$ denote $a_j\text{ }(\text{mod }e)$.
Define the \textbf{weight} $w(\bm{\lambda})$ of $\bm{\lambda}$ to be the non-negative integer
$$
w(\bm{\lambda}) = \left( \sum\limits_{j=1}^r c_{\overline{a_j}}(\bm{\lambda})\right) - \frac{1}{2} \sum_{i \in I} (c_{i}(\bm{\lambda})-c_{i+1}(\bm{\lambda}))^2.
$$
\end{defin}

It is also useful to define the hub of a multipartition $\bm{\lambda}$. For each $i\in I$ and $j \in \{1, \ldots, r\}$, define 
\begin{align*}
\delta_i^j(\bm{\lambda}) = & (\text{the number of removable nodes of $[\lambda^{(j)}]$ of residue $i$})\\
 &- (\text{the number of addable nodes of $[\lambda^{(j)}]$ of residue $i$}),
\end{align*}
and put $\delta_i(\bm{\lambda}) = \sum_{j=1}^r \delta_i^j(\bm{\lambda})$. The collection $(\delta_i(\bm{\lambda}) \text{ }|\text{ }i \in I)$ of integers is called the \textbf{hub} of $\bm{\lambda}$.

Moreover, we want to highlight an important feature of the weight and hub of a multipartition: they are invariants of the block containing $\bm{\lambda}$, and in fact determine the block.

\begin{prop}\cite[Proposition 3.2 \& Lemma 3.3]{Fay06}\label{hubblock}
Suppose $\bm{\lambda}$ is a multipartition of $n$ and $\bm{\mu}$ is a multipartition of $m$. Then
\begin{enumerate}
\item if $\bm{\lambda}$ and $\bm{\mu}$ have the same hub, then $m \equiv n\text{ }(\mathrm{mod}\text{ }e)$, and
\begin{equation*}
w(\bm{\lambda}) - w(\bm{\mu}) = \dfrac{r(n-m)}{e};
\end{equation*}
\item if $n = m$, then $\bm{\lambda}$ and $\bm{\mu}$ lie in the same block of $\mathcal{H}_{r,n}$ if and only if they have the same hub.
\end{enumerate}
\end{prop}

In view of this result, we may define the \textbf{hub of a block} $B$ to be the hub of any multipartition $\bm{\lambda}$ in $B$, and we write $\delta_i(B) = \delta_i(\bm{\lambda})$.

%\subsection{Computing the weight from the abacus}
Now, we summarise some results from \cite{Fay06} that shows how to compute weights of multipartitions efficiently from an abacus
display.

\begin{prop}\cite[Corollary 3.4]{Fay06}\label{mult_remhook}
Suppose $\bm{\lambda}= (\lambda^{(1)}, \ldots, \lambda^{(r)})$ is a multipartition, and that $\bm{\lambda}^-$ is the multipartition obtained from $\bm{\lambda}$ by removing an $e$-rim hook from some $\lambda^{(j)}$. Then $w(\bm{\lambda}) = w(\bm{\lambda^-}) + r$.
\end{prop}

\begin{defin}
If $\bm{\lambda} = (\lambda^{(1)}, \dots, \lambda^{(r)})$ is a multipartition and each $\lambda^{(j)}$ is an $e$-core, then we say that $\bm{\lambda}$ is a \textbf{multicore}.
\end{defin}

%\begin{prop}\cite[Proposition 3.5]{Fay06}\label{multcore_weight}
%Suppose that  $\bm{\lambda}= (\lambda^{(1)}, \ldots, \lambda^{(r)})$ is a multicore. Then
%$$
%w(\bm{\lambda})=\sum_{1\leq k<l \leq r} w((\lambda^{(k)}, \lambda^{(l)})).
%$$
%\end{prop}
Let $i\in I$. Suppose $\bm{\lambda}= (\lambda^{(1)}, \ldots, \lambda^{(r)})$ is a multicore. We construct the corresponding abacus display for $\bm{\lambda}$ as in Section \ref{betaAbacus}. For $j,k \in \{1,\dots, r\}$, we define
$$
\gamma_i^{jk}(\bm{\lambda})=\ell^{\bm{\mathrm{a}}}_{ij}(\bm{\lambda})-\ell^{\bm{\mathrm{a}}}_{ik}(\bm{\lambda})).
$$
$\gamma_i^{jk}(\bm{\lambda})$ may then be regarded as the difference in height between the lowest bead on runner $i$ of the abacus display for $\lambda^{(j)}$ and the lowest bead on runner $i$ of the abacus display for $\lambda^{(k)}$. Since $e$ is finite, the integers $\gamma_i^{jk}(\bm{\lambda})$ depend on the choice of $\bm{\mathrm{a}}$ if we change any $a_j$ by a multiple of $e$, but the differences
$$
\gamma_{il}^{jk}(\bm{\lambda}):=\gamma_i^{jk}(\bm{\lambda})-\gamma_l^{jk}(\bm{\lambda})
$$
do not.

Suppose $\bm{\lambda}= (\lambda^{(1)}, \ldots, \lambda^{(r)})$ is a multicore, $i,l\in I$ and $j,k \in \{1,\dots, r\}$. We define $s_{il}^{jk}(\bm{\lambda})$ to be the multicore whose abacus configuration is obtained from that of $\bm{\lambda}$ by moving a bead from runner $i$ to runner $l$ on the abacus for $\bm{\lambda}^{(j)}$, and moving a bead from runner $l$ to runner $i$ on the abacus for $\bm{\lambda}^{(k)}$. It is worth noting that $s^{jk}_{il}(\bm\lambda)=s^{kj}_{li}(\bm\lambda)$ for all $i,j,k,l$.

\begin{prop}\cite[Proposition 1.6]{Fay07}\label{computew}
Let $\bm{\lambda}$ be a multicore and let $s_{il}^{jk}(\bm{\lambda})$ be defined as above. Then
\begin{enumerate}
\item $s_{il}^{jk}(\bm{\lambda})$ has the same hub as $\bm\lambda$, and,
\item $w(s_{il}^{jk}(\bm{\lambda}))=w(\bm{\lambda})- r(\gamma_{il}^{jk}(\bm{\lambda})-2).$
\end{enumerate}
\end{prop}

\subsection{Scopes isometries}
Here we introduce maps between blocks of Ariki-Koike algebras analogous to those defined by Scopes \cite{Scopes91} between blocks of symmetric groups. Suppose $i \in I$, and let
$\phi_i \colon \mathbb{Z} \rightarrow \mathbb{Z}$ be the map given by
\begin{equation*}
\phi_i(x)=
\begin{cases}
x + 1 & x \equiv i - 1 \quad(\mathrm{mod}\text{ }e)\\
x - 1 & x \equiv i \quad(\mathrm{mod}\text{ }e)\\
x & \text{otherwise.}
\end{cases}
\end{equation*}
%If $e$ is finite, then 
$\phi_i$ descends to give a map from $I$ to $I$; we abuse notation by referring to this map as $\phi_i$ also.\\

Now suppose $\bm{\lambda}$ is a multipartition, and consider its abacus display with respect to the multicharge $\bm{\mathrm{a}}=(a_1, \ldots, a_r)$.
For each $j$, we define a partition $\Phi_i(\lambda^{(j)})$ by replacing each beta-number $\beta$ with $\phi_i(\beta)$.
Equivalently, we simultaneously remove all removable nodes of residue $i$ from $[\lambda^{(j)}]$ and add all addable nodes of $[\lambda^{(j)}]$ of residue $i$. If $i\neq 0$, this is equivalent to swapping runners $i-1$ and $i$ of each component in the abacus display of $\bm\lambda$; {if $i=0$, we swap runners $0$ and $e-1$ and then increase the level of each bead on runner $0$ by $1$ and decrease the level of each bead on runner $e-1$ by $1$.} We define $\Phi_i(\bm{\lambda})$ to be the multipartition $(\Phi_i(\lambda^{(1)}), \ldots,\Phi_i(\lambda^{(r)}))$.

\begin{exe}
Here we give two examples of the definition of $\phi_i$ for a partition. For $e=3$, consider the abacus configuration of the partition $\lambda$ given below and apply to $\lambda$ the map $\phi_1$. Then,
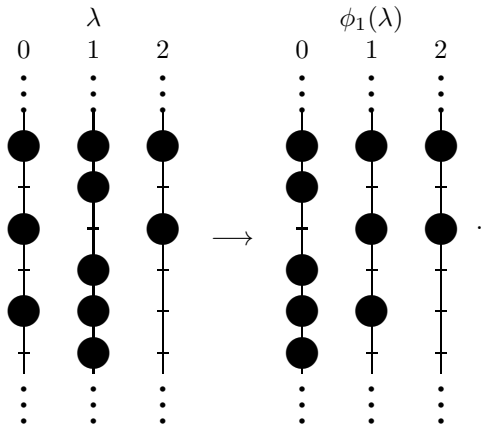
\begin{center}
\begin{tabular}{ccc}
$\lambda$ &  & $\phi_1(\lambda)$ \\
$       \begin{matrix}

        0 & 1 & 2 \\

        \bigvd & \bigvd & \bigvd \\

        \bigbd & \bigbd & \bigbd \\

        \bignb & \bigbd & \bignb \\

        \bigbd & \bignb & \bigbd \\
        
        \bignb & \bigbd & \bignb \\
        
        \bigbd & \bigbd & \bignb \\
        
        \bignb & \bigbd & \bignb \\

        \bigvd & \bigvd & \bigvd \\

        \end{matrix}
$
&
$\longrightarrow$
&
$       \begin{matrix}

        0 & 1 & 2 \\

        \bigvd & \bigvd & \bigvd \\

        \bigbd & \bigbd & \bigbd \\

        \bigbd & \bignb & \bignb \\

        \bignb & \bigbd & \bigbd \\
        
        \bigbd & \bignb & \bignb \\
        
        \bigbd & \bigbd & \bignb \\
        
        \bigbd & \bignb & \bignb \\

        \bigvd & \bigvd & \bigvd \\

        \end{matrix}
$\\
\end{tabular}.
\end{center}
So, it can be seen that the removal and the addition of nodes for $\lambda$ due to the application of $\phi_1$ is simultaneous.

Now consider again the above abacus configuration for the partition $\lambda$, but this time we apply $\phi_0$ to it. Then,
\begin{center}
\begin{tabular}{ccc}
$\lambda$ &  & $\phi_0(\lambda)$ \\
$       \begin{matrix}

        0 & 1 & 2 \\

        \bigvd & \bigvd & \bigvd \\

        \bigbd & \bigbd & \bigbd \\

        \bignb & \bigbd & \bignb \\

        \bigbd & \bignb & \bigbd \\
        
        \bignb & \bigbd & \bignb \\
        
        \bigbd & \bigbd & \bignb \\
        
        \bignb & \bigbd & \bignb \\

        \bigvd & \bigvd & \bigvd \\

        \end{matrix}
$
&
$\longrightarrow$
&
$       \begin{matrix}

        0 & 1 & 2 \\

        \bigvd & \bigvd & \bigvd \\

        \bigbd & \bigbd & \bignb \\

        \bigbd & \bigbd & \bigbd \\

        \bignb & \bignb & \bignb \\
        
        \bigbd & \bigbd & \bigbd \\
        
        \bignb & \bigbd & \bignb \\
        
        \bignb & \bigbd & \bignb \\

        \bigvd & \bigvd & \bigvd \\

        \end{matrix}
$\\
\end{tabular}.
\end{center}

\end{exe}

\begin{prop}\cite[Proposition 4.6]{Fay06}\label{bij_mpart}
Suppose $B$ is a block of $\mathcal{H}_{r,n}$ and $i \in I$. Then there is a block $\bar{B}$ of $\mathcal{H}_{r,n-\delta_i(B)}$ with the same weight as $B$. Moreover, $\Phi_i$ gives a bijection between the set of multipartitions in $B$ and the set of multipartitions in $\bar{B}$.
\end{prop}

We write $\Phi_i(B)$ for the block $\bar{B}$ described in Proposition \ref{bij_mpart}.

\subsection{Core blocks of Ariki-Koike algebras}
Here we introduce the notion of core blocks of Ariki-Koike algebras, giving several equivalent definitions following the work of Fayers in \cite{Fay07}. This definition is the one generalising the  $e$-core notion to multipartitions.

\begin{thm}\label{corebleq}
Suppose that $\bm{\lambda}$ is a multipartition lying in a block $B$ of $\mathcal{H}_{r,n}$. Let $\kappa$ be a multicharge of $\mathcal{H}_{r,n}$. The following are equivalent.
\begin{enumerate}
\item $\bm{\lambda}$ is a multicore, and there exists a multicharge $\bm{\mathrm{a}} = (a_1, \ldots, a_r)$ such that $a_j \equiv \kappa_j\text{ }(\mathrm{mod}\text{ }e)$ for all $j$ and integers $\alpha_0, \ldots, \alpha_{e-1}$ such that for each $i\in I$ and $j\in \{1, \ldots, r\}$, $\ell^{\bm{\mathrm{a}}}_{ij}(\bm{\lambda})$ equals either $\alpha_i$ or $\alpha_i + 1$.
We call such an $e$-tuple $(\alpha_0, \ldots, \alpha_{e-1})$ a \textbf{base tuple} for $\bm{\lambda}$.
%\item $\bm{\lambda}$ is a multicore, and there exist a multicharge $\bm{\mathrm{a}} = (a_1, \ldots, a_r)$ such that $a_i \equiv \kappa_i\text{ }\mathrm{mod}\text{ }e$ and integers $s_1, \ldots, s_r$ such that
%$$
%\dfrac{\mathfrak{b}^{\bm{\mathrm{a}}}_{ij}(\bm{\lambda}) - \mathfrak{b}^{\bm{\mathrm{a}}}_{ik}(\bm{\lambda})}{e} \leq s_j - s_k + 1
%$$
%for all $i \in \{0, \ldots, e-1\}$, $j, k \in \{1, \ldots, r\}$.
%\item  $\bm{\lambda}$ is a multicore, and for any multicharge $\bm{\mathrm{a}} = (a_1, \ldots, a_r)$ such that $a_i \equiv \kappa_i\text{ }\mathrm{mod}\text{ }e$ there exist integers $s_1, \ldots, s_r$ such that
%$$
%\dfrac{\mathfrak{b}^{\bm{\mathrm{a}}}_{ij}(\bm{\lambda}) - \mathfrak{b}^{\bm{\mathrm{a}}}_{ik}(\bm{\lambda})}{e} \leq s_j - s_k + 1
%$$
%for all $i \in \{0, \ldots, e-1\}$, $j, k \in \{1, \ldots, r\}$.
\item There is no block of any $\mathcal{H}_{r,m}$ with %the same multicharge $\kappa$ of $\mathcal{H}_{r,n}$, 
the same hub as $B$ and smaller weight than $B$.
\item Every multipartition in $B$ is a multicore.
\end{enumerate}
\end{thm}
Now we can make the definition of a core block for the Ariki-Koike algebra $\mathcal{H}_{r,n}$.
\begin{defin}
Suppose $B$ is a block of $\mathcal{H}_{r,n}$. Then we say that $B$ is a \textbf{core block} if and only if %either
%\begin{itemize}
%\item
the equivalent conditions of Theorem \ref{corebleq} are satisfied for any multipartition $\bm{\lambda}$ in $B$, %or
%\item $e = \infty$.
%\end{itemize}
\end{defin}

{Theorem \ref{corebleq} gives us several equivalent conditions for a multipartition to lie in a core block. %As a corollary, we can generate all the multipartitions lying in a given core block by a sequence of `elementary moves'. 
Moreover, condition (2) of Theorem \ref{corebleq} together with point (2) of Proposition \ref{hubblock} implies that, of the blocks with a given hub, only the one with the smallest weight is a core block. So, if $\bm\lambda$ is a multipartition with this hub, then we may speak of this core block as the \textbf{core block} of $\bm\lambda$.}

Adapting Theorem \ref{corebleq} we have the following result.

\begin{prop}\label{charCoreBl}
%Suppose $e < \infty $, 
Suppose $\bm{\lambda}$ is a multicore and $\kappa =(\kappa_1, \ldots, \kappa_r)$ is a multicharge for $\mathcal{H}_{r,n}$. Then $S^{\bm{\lambda}}$ lies in a core block of $\mathcal{H}_{r,n}$ if and only if there is $\bm{\mathrm{a}} = (a_1, \ldots, a_r) \in \mathbb{Z}^r$ such that $a_j \equiv \kappa_j \text{ }(\mathrm{mod}\text{ }e)$ for all $j$ and an abacus configuration for $\bm{\lambda}$ such that
$$|\gamma^{jk}_i (\bm{\lambda})| = |\ell^{\bm{\mathrm{a}}}_{ij}(\bm{\lambda}) - \ell^{\bm{\mathrm{a}}}_{ik}(\bm{\lambda})| \leq 1 \text{ for each } i \in I \text{ and } j,k \in \{1, \ldots, r\}.$$
\end{prop}

%\begin{rem}
%If $e$ is finite and $\bm\lambda$ is a multicore, then the integers $\gamma^{jk}_i (\bm{\lambda})$ are the difference of beads between the runner $i$ of the abacus display for $\lambda^{(j)}$ and those on runner $i$ of the abacus display for $\lambda^{(k)}$ and depend on the choice of $\bm{\mathrm{a}}$.
%However, the differences
%$$\gamma_{il}^{(jk)}(\bm{\lambda}):=\gamma_i^{(jk)}(\bm{\lambda})-\gamma_l^{(jk)}(\bm{\lambda})$$
%depend on $\bm{\mathrm{a}}$.
%\end{rem}

%Now, denote the difference of beads between two runners of truncated abacus display by
%$$\gamma^{(j)}_{il} (\bm{\lambda}):=\Gamma^{\bm{\mathrm{a}}}_{ij} (\bm{\lambda}) - \Gamma^{\bm{\mathrm{a}}}_{lj}(\bm{\lambda}),$$
%for all $i,l \in \{0, 1, \ldots e-1\}$ and $j \in \{1, \ldots, r\}$.

\begin{cor}\label{k-m<=2}
Assume the conditions of Proposition \ref{charCoreBl}. If $S^{\bm{\lambda}}$ lies in a core block of $\mathcal{H}_{r,n}$, then there exists $\bm{\mathrm{a}} = (a_1, \ldots, a_r) \in \mathbb{Z}^r$ such that $a_j \equiv \kappa_j \text{ }(\mathrm{mod}\text{ }e)$ for all $j$ so that the abacus configuration of $\bm{\lambda}$ with respect to $\bm{\mathrm{a}}$ is such that
$$|\delta^{j}_{i}(\bm{\lambda}) - \delta^{k}_{i}(\bm{\lambda}) |\leq 2 \text{ for each } i \in I \text{ and } j,k \in \{1, \ldots, r\}.$$
\end{cor}

\begin{proof}
By Proposition \ref{charCoreBl}, since there exists a multicharge $\mathrm{a}$ such that $|\gamma^{jk}_i (\bm{\lambda})| \leq 1$ for each $i \in I$ and $j,k \in \{1, \ldots, r\}$ we have
\begin{align*}
|\delta^{j}_{i}(\bm{\lambda}) - \delta^{k}_{i}(\bm{\lambda}) | &=|\ell^{\bm{\mathrm{a}}}_{ij}- \ell^{\bm{\mathrm{a}}}_{i-1,j} -(\ell^{\bm{\mathrm{a}}}_{ik}- \ell^{\bm{\mathrm{a}}}_{i-1,k})|\\
&= |\ell^{\bm{\mathrm{a}}}_{ij}- \ell^{\bm{\mathrm{a}}}_{ik} +\ell^{\bm{\mathrm{a}}}_{i-1,j}- \ell^{\bm{\mathrm{a}}}_{i-1,k}|\\
&\leq |\ell^{\bm{\mathrm{a}}}_{ij}- \ell^{\bm{\mathrm{a}}}_{ik}| +|\ell^{\bm{\mathrm{a}}}_{i-1,j}- \ell^{\bm{\mathrm{a}}}_{i-1,k}|\\
&\leq1+1=2.
\end{align*}
\end{proof}

%\todo{Consider moving this section to later}
%There is a way such that one may get from a multipartition to any other multipartition in the same block by a sequence of 'elementary moves', as described in the following result.
%Given multipartitions $\bm\lambda$ and $\bm\mu$, we write $\bm\lambda \leftrightsquigarrow \bm\mu$ (and say that $\bm\mu$ is obtained from $\bm\lambda$ by an \textit{elementary move}) if one of the following holds:
%\begin{itemize}
%\item $[\bm\mu]$ is obtained from $[\bm\lambda]$ by adding or removing a rim $e$-hook,
%\item $\bm\lambda$ and $\bm\mu$ are both multicores, and $\bm\mu = s_{il}^{jk}(\bm\lambda)$ for some $i$, $l$, $j$, $k$.
%\end{itemize}

%\begin{prop}\cite[Theorem 3.7]{Fay07}\label{blockmoves}
%Suppose $\bm\lambda$ and $\bm\mu$ are multipartitions lying in the same block of $\mathcal{H}_{n,r}$. Then there is a sequence $\bm\lambda = \bm\lambda_0,\ldots, \bm\lambda_t = \bm\mu$ of multipartitions such that $\bm\lambda_{i-1} \leftrightsquigarrow \bm\lambda_i$ for each $i$, and $w(\bm\lambda_i) \leq w(\bm\lambda)$ for each $i$.
%\end{prop}

\section{Results about multicores}\label{multicoreresults}
In this section, we give some results concerning properties of multicores that play a fundamental role in the proof of the main result of this paper. We fix some notation.

Let $\mathcal{H}_{r,n}$ be an Ariki-Koike algebra with multicharge $\kappa$. For $i\in I$ and a multicore $\bm{\mu}$, denote by:
$$
d_i(\bm \mu)=\min\{\delta_{i}^j(\bm \mu) \text{ }|\text{ } j \in \{1, \ldots, r\}\}.
$$ 
If the value of $i$ is clear, we will write $d(\bm \mu)$ instead of $d_i (\bm \mu)$.

%We should be able to reproduce the proof of Lemma 2.1 for $e=2$ in the general case, that is for any $e$ just doing some small adjustments.

%We can rephrase Proposition \ref{computew} in the following way that will be helpful below.
%\begin{rem}\label{rem_sijk}
%Let $\bm \lambda$ be a multicore in $\mathcal{H}_{r,n}$, $i\in \mathbb{Z}/e\mathbb{Z}$ and $j,k\in\{1,\ldots, r\}$. Then:
%\begin{itemize}
%\item[1)] $s_{i,i-1}^{jk}(\bm\lambda)$ has the same hub as $\bm\lambda$;
%\item[2)] $$ \delta_i^{l}(s_{i,i-1}^{jk}(\bm\lambda))=
%\begin{cases}
%\delta_i^l(\bm \lambda) & \text{ if } l\neq j,k \\
%\delta_i^l(\bm \lambda)-2 & \text{ if } l=j \\
%\delta_i^l(\bm \lambda)+2 & \text{ if } l=k.
%\end{cases}$$
%\item[3)] $w(s_{i,i-1}^{jk}(\bm\lambda))=w(\bm\lambda)-r(\delta_i^j(\bm\lambda)-\delta_i^k(\bm\lambda)-2)$.
%\end{itemize}
%\end{rem}

Firstly, we notice an important and useful property of the abacus display of a multicore $\bm \mu$ lying in a core block $C$ of $\mathcal{H}_{r,n}$ and then we give some results about multicores not necessarily in a core block.

If $\bm\mu$ is a multipartition lying in a core block $C$ of $\mathcal{H}_{r,n}$ then, by the definition of a base tuple, there exists a multicharge $\bm{\mathrm{a}} = (a_1, \ldots, a_r)$ of $\mathcal{H}_{r,n}$ and at least one base tuple $(b_0, \ldots, b_{e-1})$ such that $\bm\mu$ has abacus display where for any $i \in I\setminus\{0\}$ and $j\in\{1, \ldots, r\}$
\begin{equation}\label{djivalues}
\delta^{j}_{i}(\bm{\mu}) \in \{b_i-b_{i-1}-1, b_i-b_{i-1}, b_i-b_{i-1}+1\},
\end{equation}
and, for $i=0$ and $j\in\{1, \ldots, r\}$
\begin{equation}\label{djivalues0}
\delta^{j}_{i}(\bm{\mu}) \in \{b_0-b_{e-1}-2, b_0-b_{e-1}-1, b_0-b_{e-1}\}.
\end{equation}
Notice that \eqref{djivalues} and \eqref{djivalues0} hold for all the multicores in the core block $C$ since the base tuple is an invariant of a core block.

Let $C$ be a core block. Fix $i \in I$ and a multicharge $\bm{\mathrm{a}} = (a_1, \ldots, a_r)$ such that \eqref{djivalues} and \eqref{djivalues0} hold. Let $(b_0,\ldots, b_{e-1})$ be the corresponding base tuple such that ${i}$ is as big as possible and $b_{i-1}$ is as small as possible between all the possible choices of base tuples corresponding to $\bm{\mathrm{a}}$. Then we can define
\begin{equation}\label{Kidef}
K_{{i}}=K_{{i}}(C) := \begin{cases}
b_{{i}}-b_{{i}-1}-1 \quad \text{if } {i} \in I\setminus \{0\};\\
b_{0}-b_{e-1}-2 \quad \text{if } {i}=0.
\end{cases}
%\text{ for any multicore }\bm\mu \in C.
\end{equation}
%Hence, the integer $K_{\bar{i}}$ is well defined for a core block.
If it is clear what ${i}$ we are referring to, we will simply write $K=K(C)$ instead of $K_{{i}}(C)$. Note that 
\begin{equation}\label{K_leqd_mu}
K\leq d(\bm\mu)
\end{equation}
for all $\bm\mu \in C$.

\begin{exe}\label{exKi}
Suppose $e=5$, $r=3$ and $\kappa=(0,3,1)$. Let $\bm\mu$ be the multicore $((4,3,1),(4,2^3),(3,2))$ which has abacus display with respect to the multicharge $\bm{\mathrm{a}}=(0,-2,1)$
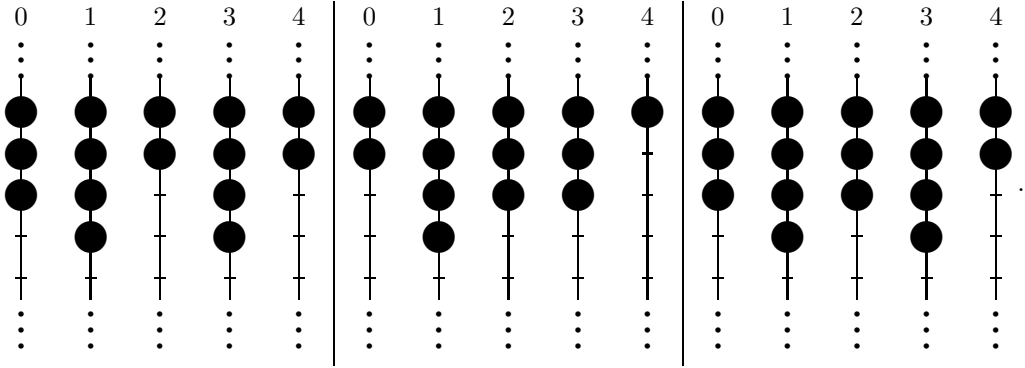
\begin{center}
\begin{tabular}{c|c|c}
$      
        \begin{matrix}
        
        0 & 1 & 2 & 3 & 4 \\
        
        \bigvd & \bigvd & \bigvd & \bigvd & \bigvd \\

        \bigbd & \bigbd & \bigbd & \bigbd & \bigbd \\
        
        \bigbd & \bigbd & \bigbd & \bigbd & \bigbd \\
        
        \bigbd & \bigbd & \bignb & \bigbd & \bignb \\
        
        \bignb & \bigbd & \bignb & \bigbd & \bignb \\

        \bignb & \bignb & \bignb & \bignb & \bignb \\
        
        \bigvd & \bigvd & \bigvd & \bigvd & \bigvd
        
        \end{matrix}
$
&
$      
        \begin{matrix}

        0 & 1 & 2 & 3 & 4 \\

        \bigvd & \bigvd & \bigvd & \bigvd & \bigvd \\

        \bigbd & \bigbd & \bigbd & \bigbd & \bigbd \\
        
        \bigbd & \bigbd & \bigbd & \bigbd & \bignb \\
        
        \bignb & \bigbd & \bigbd & \bigbd & \bignb \\
        
        \bignb & \bigbd & \bignb & \bignb & \bignb \\

        \bignb & \bignb & \bignb & \bignb & \bignb \\
        
        \bigvd & \bigvd & \bigvd & \bigvd & \bigvd
        
        \end{matrix}
$
&
$      
        \begin{matrix}
        
         0 & 1 & 2 & 3 & 4 \\

        \bigvd & \bigvd & \bigvd & \bigvd & \bigvd \\

        \bigbd & \bigbd & \bigbd & \bigbd & \bigbd \\
        
        \bigbd & \bigbd & \bigbd & \bigbd & \bigbd \\
        
        \bigbd & \bigbd & \bigbd & \bigbd & \bignb \\
        
        \bignb & \bigbd & \bignb & \bigbd & \bignb \\

        \bignb & \bignb & \bignb & \bignb & \bignb \\
        
        \bigvd & \bigvd & \bigvd & \bigvd & \bigvd
        
        \end{matrix}
$.
\end{tabular}
\end{center}
Since $|\gamma^{jk}_{i}(\bm\mu)| \leq 1$ for all $j,k \in \{1, 2, 3\}$ and $i\in \{0, \ldots, 4\}$, $\bm \mu$ lies in a core block $C$ by Proposition \ref{charCoreBl}. Then we can define a base tuple for $\bm\mu$ and thus for $C$. In particular, we can consider the following two base tuples:
\begin{enumerate}
\item $(b_0, b_1, b_2, b_3, b_4)=(2,4,2,3,1)$;
\item $(b_0', b_1', b_2', b_3', b_4')=(2,3,2,3,1)$.
\end{enumerate}
Take $\bar{i}=0$. In order to define $K_0$  can choose either $(b_0, b_1, b_2, b_3, b_4)$ or $(b_0', b_1', b_2', b_3', b_4')$ because $b_0=b_0'$ and  $b_4=b_4'$, so we get $K_0=-1$.\\
Take $\bar{i}=1$. In order to define $K_1$ we need to choose $(b_0, b_1, b_2, b_3, b_4)=(2,4,2,3,1)$ because $b_1>b_1'$ and  $b_0=b_0'$, so we get $K_1=1$.\\
Take $\bar{i}=3$. In order to define $K_3$ we can choose either $(b_0, b_1, b_2, b_3, b_4)$ or $(b_0', b_1', b_2', b_3', b_4')$ because $b_3=b_3'$ and  $b_2=b_2'$, so we get $K_3=0$.
\end{exe}

%We will explain later the reason of this choice for the base tuple.

We now want to exhibit, for all multicores of $\mathcal{H}_{r,n}$, a sequence of multicores of non-increasing weight ending in a multicore lying in a core block. In order to do this, we need the following lemma adapted from the proof of Proposition 3.7 in \cite{Fay07} to the case of multicores.

\begin{lem}\label{seq_Fay}
If $\bm m$ is a multicore in a block of $\mathcal{H}_{r,n}$, then there is a sequence $\bm m=\bm m_0,\ldots, \bm m_u$ of multicores such that, for all $t=0,\ldots, u-1$, we have $\bm m_{t+1}=s_{il}^{jk}(\bm m_{t})$ for some $i,l,j,k$ and $\bm m_u$ lies in the core block of $\bm m$. Furthermore, $w(\bm m_{t+1})\leq w(\bm m_{t})$ for all $t=0,\ldots, u-1$. %and $w(\bm\lambda_u)<w(\bm\lambda_{u-1})$. 
\end{lem}

\begin{prop}\label{seq_mult_anye}
Let $\bm m$ be a multicore of $\mathcal{H}_{r,n}$. Then there exist a multicore $\bm\mu$ in the core block $C$ of $\bm m$ and a sequence of multicores
$$\bm m = \bm m_0, \bm m_1, \ldots, \bm m_{s-1}, \bm m_{s}= \bm \mu$$
such that:
\begin{itemize}
\item[1)] the core block of $\bm m_t$ is $C$ for all $t=0,\ldots, s$;

\item[2)] $\bm m_{t+1}=s^{jk}_{il}(\bm m_t)$ for some $j,k\in\{1,\ldots ,r\}$, $i,l \in \{0,1, \ldots, e-1\}$ for all $t=0,\ldots, s-1$;

\item[3)] there exists $0 \leq v \leq s$ such that 
\begin{itemize}
\item[i)] 
$w(\bm m_{t+1})< w(\bm m_t)$ for all $t=0,\ldots, v-1$ and $w(\bm m_{t+1})\leq w(\bm m_t)$ for all $t=v,\ldots, s-1$;
\item[ii)] $|\gamma_{il}^{jk}(\bm m_v)| \leq 2$ for all $i,l$ and $j,k$.
\end{itemize}
\end{itemize}
%We say that $\bm \mu$ is the multicore in the core block $C$ corresponding to $\bm m$.
\end{prop}

\begin{proof}
%\todo{More formally this proof should be done with induction on $w(\lambda) - w(C)$.}
%\todo{We should give another characterization of the core block in terms of $\gamma_{il}^{jk}\leq 2$ or something similar}
Let $\bm m_0 := \bm m$ be a multicore of $\mathcal{H}_{r,n}$. 
Then, apply the following procedure for $t\geq 0$.
\begin{enumerate}
\item[1.] Calculate $\gamma_i^{jk}(\bm m_t)$ for all $i\in\{0, \ldots, e-1\}$ and $j,k \in \{1, \ldots, r\}$;
\item[2.] If there is a choice of $i,l$ and $j,k$ such that $\gamma_{il}^{jk}(\bm m_t)\geq 3$, set $\bm m_{t+1}=s_{il}^{jk}(\bm m_t)$. By Proposition \ref{computew} we have 
$$w(\bm m_{t+1})< w(\bm m_t).$$
\item[3.] Repeat this step until we have $\bm m_{t+1}$ with $\gamma_{il}^{jk}(\bm m_{t+1}) \leq 2$ for all $i,l$ and $j,k$.
\end{enumerate}

Suppose that we stop for $t+1=v$. Notice that $\gamma_{il}^{jk}(\bm m_v) \leq 2$ for all $i,l$ and $j,k$ implies $|\gamma_{il}^{jk}(\bm m_v)| \leq 2$ for all $i,l$ and $j,k$. Indeed, $$\gamma_{il}^{jk}(\bm m_v) = -\gamma_{li}^{jk}(\bm m_v)= - \gamma_{il}^{kj}(\bm m_v).$$

{Now, if $\bm m_v$ is not in the core block $C$, apply Lemma \ref{seq_Fay} until we get a multicore $\bm\mu$ in the core block $C$.}
%Now, let $w(\bm m_v)=w(C)+hr$ with $h\geq 0$. Then, proceeding by induction on $h$ we have the following.
%\begin{enumerate}
%\item If $h=0$, then $w(\bm m_v)=w(C)$; therefore $\bm m_v$ is in the core block $C$ and so $\bm\mu:=\bm m_v$.
%\item If $h>0$, then by Lemma \ref{seq_Fay} there is a sequence of multicores $\bm m_v, \bm m_{v+1}, \ldots, \bm m_u$ such that $\bm m_{t+1}=s_{il}^{jk}(\bm m_t)$ for some $i,l,j,k$, $w(\bm m_{t+1})\leq w(\bm m_t)$ for $t=v, \ldots, u-2$ and $w(\bm m_{u})< w(\bm m_{u-1})$. Therefore, $w(\bm m_{u})=w(C)+h'r$ for some $0 \leq h' < h$. Hence, by induction hypothesis there exists a sequence of multicores $\bm m_u, \bm m_{u+1}, \ldots, \bm m_s$ with $\bm m_s$ in the core block $C$ and $w(\bm m_{t+1})\leq w(\bm m_t)$ for all $t=u, \ldots, s-1$.
%\end{enumerate}
\end{proof}

%\begin{cor}\label{cor_seq_mult_anye}
%With the same notation as in Proposition \ref{seq_mult_anye}, let $\bm \mu =  s^{j_{s}k_{s}}_{i_{s}l_{s}}\ldots s^{j_1k_1}_{i_1l_1}(\bm m)$. Then, 
%$$\bm m = s^{j_1k_1}_{l_1i_1} \ldots s^{j_{s}k_{s}}_{l_si_s}(\bm \mu),$$
%where
%\begin{enumerate}
%\item[1)] $\widetilde{\bm m}_0:=\bm\mu$ and $\widetilde{\bm m}_t:=\bm m_{s-t}=s^{j_{s-t+1}k_{s-t+1}}_{l_{s-t+1}i_{s-t+1}} \ldots s^{j_sk_s}_{l_si_s}(\bm \mu)$ for $1\leq t \leq s$;
%\item[2)] $w(\widetilde{\bm m}_{t+1})\geq w(\widetilde{\bm m}_t)$ for all $t=0,\ldots, v-1$ and $w(\widetilde{\bm m}_{t+1})>w(\widetilde{\bm m}_{t})$ for $v\leq t \leq s-1$;
%\item[3)] $|\gamma_{il}^{jk}(\widetilde{\bm m}_v)| \leq 2$ for all $i,l$ and $j,k$.
%\end{enumerate}
%\end{cor}
%
%\begin{proof}
%This follows directly by Proposition \ref{seq_mult_anye} and by definition of $s^{jk}_{il}$.
%\end{proof}

Before stating the main result, we need some preliminary lemmas.

\begin{lem}\label{delta_interval}
Suppose that $\bm m$ is a multicore such that $|\gamma_{il}^{jk}(\bm m)|\leq 2$ for all $i,l \in I$ and $j,k\in\{1, \ldots r\}$. Fix $\bar{i}\in I$, and let $d$ be the integer such that $d(\bm m) = d-1$. Then $$\delta_{\bar{i}}^j(\bm m) \in \{d-1, d, d+1\} \text{ for all } j \in \{1, \ldots, r\}.$$ 
\end{lem}
\begin{proof}
Consider $\bm m$ a multicore such that $|\gamma_{il}^{jk}(\bm m)|\leq 2$ for all $i,l \in I$ and $j,k\in\{1, \ldots r\}$. %Denote by $\ell^{\kappa}_{ij}(\bm m)$ the level of the last bead in runner $i$ of the abacus display $m^{(j)}$ with respect to the multicharge $\kappa$.
Fix $\bar{i} \in I$. %We can find a multicharge $\bm{\mathrm{a}}$ such that $\ell^{\bm{\mathrm{a}}}_{(\bar{i}-1)j}(\bm m)=0$ for all $j \in \{1, \ldots, r\}$.
Then, since $|\gamma_{il}^{jk}(\bm m)|\leq 2$ for all $i,l \in I$ and $j,k\in\{1, \ldots r\}$, we have that
$$|\gamma_{(\bar{i}-1)\bar{i}}^{jk}(\bm m)| = |\delta_{\bar{i}}^{k}(\bm m) - \delta_{\bar{i}}^{j}(\bm m)| 
%= |\ell^{\bm{\mathrm{a}}}_{\bar{i}k}(\bm m) - \ell^{\bm{\mathrm{a}}}_{\bar{i}j}(\bm m)|
\leq 2.$$
Hence, since $d(\bm m) = d-1$, we get $\delta_{\bar{i}}^j(\bm m) \in \{d-1, d, d+1\}$ for all $j \in \{1, \ldots, r\}$.
\end{proof}

\begin{lem}\label{l'<=l+1}
Let $\bm m$ be a multicore with core block $C$ and such that $|\gamma_{il}^{jk}(\bm m)|\leq 2$ for all $i,l \in I$ and $j,k\in\{1, \ldots r\}$. Suppose that $\bm\mu$ is a multicore in the core block $C$. %corresponding to $\bm m$. 
Fix $\bar{i}\in I$, and let $d(\bm m) = d-1$ and $d(\bm \mu) = d'-1$ for some integers $d$ and $d'$. Then $d'\leq d+1$.
\end{lem}
\begin{proof}
By Lemma \ref{delta_interval}, we know that $\delta_{\bar{i}}^j(\bm m) \in \{d-1, d, d+1\}$ and $\delta_{\bar{i}}^j(\bm\mu) \in \{d'-1, d', d'+1\}$  for all $j \in \{1, \ldots, r\}$.
Let $a$, $b$, and $c$ be the number of $\delta_{\bar{i}}^j(\bm m)$ equal respectively to $d-1$, $d$, and $d+1$. Let $a'$, $b'$, and $c'$ be the number of $\delta_{\bar{i}}^j(\bm\mu)$ equal respectively to $d'-1$, $d'$, and $d'+1$. Notice that $a>0$ and $a'>0$ by definition of $d-1$ and $d'-1$. By Proposition \ref{seq_mult_anye}, we can go from the multicore $\bm m$ to the multicore $\bm\mu$ in the core block $C$ via a sequence of multicores $\bm m_t$ such that $\bm m_{t+1}=s_{il}^{jk}(\bm m_t)$ for some $i,l \in I$ and $j,k \in \{1, \ldots, r\}$. By point (1) of Proposition \ref{computew}, we know that each multicore $\bm m_t$ occuring in this sequence has the same hub of $\bm m$, then $\bm m$ and $\bm \mu$ have the same hub. So,
\begin{equation}\label{samehub}
a(d-1)+b(d)+c(d+1)=a'(d'-1)+b'(d')+c'(d'+1),
\end{equation}
where $a+b+c=a'+b'+c'=r$.
Suppose for a contradiction that $d'>d+1$. Then, looking at \eqref{samehub} we have
\begin{align*}
\text{LHS} &\leq r(d+1) \text{ with equality if and only if } a=b=0;\\
\text{RHS} &\geq a'(d+1)+b'(d+2)+c'(d+3)\geq r(d+1) \text{ with equality if and only if } b'=c'=0;
\end{align*}
We must have equality in both terms, but this is a contradiction since $a>0$.
\end{proof}

%\begin{cor}\todo{Maybe it is worth defining the minimum with d instead of using l and l'}
%In the same notation of Lemma \ref{l'<=l+1}, set 
%\begin{align*}
%d &:= \min\{\delta_i^j(\bm m)\text{ }|\text{ }j \in \{1, \ldots, r\}\};\\
%d' &:= \min\{\delta_i^j(\bm\mu)\text{ }|\text{ }j \in \{1, \ldots, r\}\}.
%\end{align*}
%Then $d' \leq d+1$. 
%\end{cor}
%
%\begin{proof}
%This is a consequence of Lemma \ref{l'<=l+1}. Indeed, if $l'\leq l+1$, then $d'=l'-1 \leq l = d+1$.
%\end{proof}

\begin{lem}\label{d>=d'-2}
Let $\bm m$ be a multicore of $\mathcal{H}_{r,n}$ and $\bm m'= s_{il}^{jk}(\bm m)$ for some $i,l,j,k$. Fix $\bar{i}\in I$. Then $d(\bm m')\geq d(\bm m)-2$. Moreover, if $\gamma_{il}^{jk}(\bm m) = 1$, then $d(\bm m')\geq d(\bm m)-1$.
\end{lem}
\begin{proof}
The fact that $d(\bm m')\geq d(\bm m)-2$ follows from the definition of $s_{il}^{jk}$ and that $|\delta_{\bar{i}}^j(\bm m)-\delta_{\bar{i}}^j(\bm m')|=2$ if and only if $\{i,l\}=\{\bar{i}-1,\bar{i}\}$.

Suppose $\gamma_{il}^{jk}(\bm m) = 1$. Then,
\begin{itemize}
\item if $\{i,l\}\neq \{\bar{i}-1,\bar{i}\}$, then $d(\bm m')\geq d(\bm m)-1$ since the only case in which $\delta_{\bar{i}}^j(\bm m')$ decreases by $2$ with respect to $\delta_{\bar{i}}^j(\bm m)$ is when $\{i,l\}=\{\bar{i}-1,\bar{i}\}$;
\item if $\{i,l\}=\{\bar{i}-1,\bar{i}\}$, then 
$$\gamma_{\bar{i}-1,\bar{i}}^{jk}(\bm m) = 1 \Leftrightarrow \delta_{\bar{i}}^k(\bm m)-\delta_{\bar{i}}^j(\bm m)=1 \Leftrightarrow \delta_{\bar{i}}^k(\bm m)=\delta_{\bar{i}}^j(\bm m)+1.$$
Thus,
$\delta_{\bar{i}}^k(\bm m')=\delta_{\bar{i}}^k(\bm m)+2$ and $\delta_{\bar{i}}^k(\bm m')=\delta_{\bar{i}}^k(\bm m)-1$. Hence, 
$d(\bm m')\geq d(\bm m)-1$.
\end{itemize}
\end{proof}

\begin{lem}\label{d>=d'-h}
Let $\bm m$ be a multicore of $\mathcal{H}_{r,n}$. Suppose that $\bm m = \bm m_0, \bm m_1, \ldots, \bm m_v$ is a sequence of multicores such that $\bm m_{t+1}=s_{il}^{jk}(\bm m_{t})$ for some $i,l,j,k$ and $w(\bm m_{t+1})<w(\bm m_t)$ for all $t=0,\ldots, v-1$.  Let $w(\bm m) = w(\bm m_v)+hr$ with $h>0$. Then $d(\bm m)\geq d(\bm m_v)-h$. 
\end{lem}
\begin{proof}
%By Corollary \ref{cor_seq_mult_anye} we have a sequence 
%$$\bm m_v =  \widetilde{\bm m}_0, \widetilde{\bm m}_1, \ldots, \widetilde{\bm m}_v=\bm m$$
%such that for $0\leq t \leq v-1$
%\begin{enumerate}
%\item $\widetilde{\bm m}_{t+1}:=s^{jk}_{li}(\widetilde{\bm m}_{t})$ for some $i,l,j,k$;
%\item $w(\widetilde{\bm m}_{t+1})>w(\widetilde{\bm m}_{t})$.
%\end{enumerate}
We proceed by induction on $v$. If $v=1$, the sequence of multicores consists of $\bm m_0 = \bm m$ and $\bm m_1 = s_{il}^{jk}(\bm m_0)$ for some $i,l,j,k$ with $w(\bm m_0)=w(\bm m_1)+hr$ for $h > 0$. Note that $\bm m_0=s_{li}^{jk}(\bm m_1)$. By Proposition \ref{computew}(2), the weight of $\bm m_0$ is $w(\bm m_0)= w(\bm m_{1})-r(\gamma_{li}^{jk}(\bm m_{1})-2)$, so $h=-\gamma_{li}^{jk}(\bm m_1)+2$. Moreover, $h > 0$ and so we have that $\gamma_{li}^{jk}(\bm m_1) \leq 1$. Hence, we just need to check the following two cases. 
\begin{itemize}
\item If $\gamma_{li}^{jk}(\bm m_{1}) \leq 0$, then $d(\bm m_0) \geq d(\bm m_1)-2$ by Remark \ref{d>=d'-2} and $h=-\gamma_{li}^{jk}(\bm m_1)+2 \geq 2$. Thus,
$$d(\bm m_0) \geq d(\bm m_1)-2 \geq d(\bm m_1)+\gamma_{li}^{jk}(\bm m_0) -2 = d(\bm m_1) - h.$$
\item If $\gamma_{il}^{jk}(\bm m_0) = 1$ by Remark \ref{d>=d'-2}, then $h=1$ and $d(\bm m_0)\geq d(\bm m_1)-1$. Thus,
$$d(\bm m_0) \geq d(\bm m_1)-1 = d(\bm m_1)-h.$$
\end{itemize}

Suppose $v>1$. Let $w(\bm m)=w(\bm m_{v-1})+h'r$ with $0 \leq h' < h$ and $w(\bm m_{v-1})=w(\bm m_v)+h''r$ with $h''>0$ so that $h=h''+h'$. By induction hypothesis we know that $d(\bm m)\geq d(\bm m_{v-1})-h'$. In order to get the result we want to show that $d(\bm m)\geq d(\bm m_v)-h$. We know from the base step that that $d(\bm m_{v-1})\geq d(\bm m_v)-h''$. Thus,
$$d(\bm m) \geq d(\bm m_{v-1})-h' \geq d(\bm m_{v})-h''-h' = d(\bm m_v) -h.$$
\end{proof}

\begin{prop}\label{prop_d_K-h}
Fix $\bar{i} \in I$. Let $K=K_{\bar{i}}$ be the integer defined in \eqref{Kidef} for a core block $C$. Suppose that $\bm m$ is a multicore with core block $C$ and weight
$$w(\bm m) = w(C)+hr$$
with $0 \leq h \leq K$. Then $d(\bm m)\geq K-h$. 
\end{prop}

\begin{proof}
Let
$$
\bm m = \bm m_0, \ldots, \bm m_{v}, \bm m_{v+1}, \ldots , \bm m_{s}= \bm \mu
$$
be the sequence defined in Proposition \ref{seq_mult_anye}, where $v$ is such that $|\gamma_{il}^{jk}(\bm m_v)|\leq 2$ for all $i,l,j,k$, and $\bm\mu\in C$.

By Lemma \ref{d>=d'-h}, we have that
\begin{equation}\label{dm_dm_v-h'}
d(\bm m)\geq d(\bm m_v)-h',
\end{equation}

where $0\leq h'\leq h$ is such that $w(\bm m)=w(\bm m_v) +h'r$.

If $h'=h$, then $\bm m_v=\bm m_s=\bm \mu\in C$; therefore $d(\bm m)\geq d(\bm \mu)-h\geq K-h$ by \eqref{K_leqd_mu}.

Otherwise $h'<h$, and Lemma \ref{l'<=l+1} can be applied to get that
$$
d(\bm m_v)\geq d(\bm \mu)-1\geq d(\bm \mu)-(h-h').
$$
Combining this with \eqref{dm_dm_v-h'}, we have:
\begin{equation*}
d(\bm m)\geq d(\bm m_v)-h' \geq d(\bm \mu)-(h-h')-h'= d(\bm \mu)-h 
\geq  K-h.
\end{equation*}

\end{proof}

\section{Decomposition numbers for blocks of $\mathcal{H}_{r,n}$}\label{main_res}
Now, we want to generalise Lemma 2.1 of Scopes' paper \cite{Scopes91} to a graded version for the Ariki-Koike algebras $\mathcal{H}_{r,n}$. {Thus, we want to show that, for $i \in I$, the $v$-decomposition matrices of the blocks $B$ and $\Phi_i(B)$ are the same, }
%restriction of modules leads to a natural correspondence between the multipartitions of $n$ whose Specht modules belong to $B$ and those of $n-\delta_i(B)$ whose Specht modules belong to $\Phi_i(B)$, 
provided that %w(B)\leq \delta_{i}(C) + (h+K_i+1)r
\begin{equation}\label{condition_anye}
    w(B)\leq w(C)+K_ir
\end{equation}
where
\begin{itemize}
\item $C$ is the core block of $B$,
%\item $h$ is the integer such that $w(C)-\delta_i(C) \in \{hr+1, hr+2, \ldots, (h+1)r\}$,
\item $K_i$ is the integer defined in \eqref{Kidef}.
\end{itemize}

Notice that this condition is a block condition, i.e., it is satisfied by all the multipartitions in the block. When $i$ is understood, we write $K$ instead of $K_i$.

\begin{rem}
If $r=1$, condition \eqref{condition_anye} is equivalent to Scopes' condition for the symmetric group in \cite[Section 2]{Scopes91}.
\end{rem}

Moreover, note that Proposition \ref{mult_remhook} implies the following.
\begin{thm}
Let $\bm\mu$ be a multipartition in a core block and let $0 \leq h \leq K_i$. Let $B$ be the block containing the multipartitions obtained by adding $h$ $e$-rim hooks to $\bm\mu$. Then $B$ satisfies condition \eqref{condition_anye}. 
\end{thm}

\begin{lem}\label{noconfo-}
Fix $i \in I$. Let $B$ be a block of $\mathcal{H}_{r,n}$ such that \eqref{condition_anye} holds and $\delta_i(B)\geq 0$.
\begin{itemize}
    \item If $i\neq 0$, then in each component of every $r$-multipartition $\bm\lambda$ of $n$ such that $S^{\bm{\lambda}}$ belongs to the block $B$, there is no abacus configuration of the type $ 
        \begin{matrix}

        \bd & \nb \\

        \end{matrix}
$ in runners $i-1$ and $i$. That is, there is no $1\leq j \leq r$ with $b\equiv i-1 \ (\mathrm{mod} \ e)$ and $b\in B^{j}_{a_j}, \ b+1\notin B^{j}_{a_j}$.
   \item If $i=0$, then in each component of every $r$-multipartition $\bm\lambda$ of $n$ such that $S^{\bm{\lambda}}$ belongs to the block $B$, there is no abacus configuration of the type $ 
        \begin{matrix}

        \bd & \nb \\

        \end{matrix}
$ in runners $e-1$ and $0$. That is, there is no $1\leq j \leq r$ with $b\equiv e-1 \ (\mathrm{mod} \ e)$ and $b\in B^{j}_{a_j}, \ b+1, b+e+1\notin B^{j}_{a_j}$.
\end{itemize}
\end{lem}

\begin{proof}
First of all, notice that
\begin{equation}\label{fund_condition_anye}
w(B)-w(C)=\ell r, \hspace{0.5cm} \text{for some } \ell\in\{0,\ldots,K\}
\end{equation}
since $w(B)-w(C)$ can take as values only integral multiples of $r$. 

Now, consider the multicore $\bm m$ associated to $\bm\lambda$, that is the multicore whose abacus display is obtained from the one of $\bm\lambda$ sliding all the beads up as high as possible. Then $\bm m$ has the same core block $C$ of $\bm\lambda$ and weight $w(\bm m) = w(C)+sr$ with $0 \leq s \leq \ell$. Thus, by Proposition \ref{prop_d_K-h} we can say that
$$
d(\bm m)\geq K-s\geq \ell-s,
$$
since $\ell \leq K$ by \eqref{fund_condition_anye}.
%This inequality tells us that the multipartition $\bm \lambda$ presents no configuration of the type $ 
%        \begin{matrix}
%
%        \bd & \nb \\
%
%        \end{matrix}
% in runners $i-1$ and $i$. Indeed, 
{Moreover, in order to get $\bm \lambda$ from $\bm m$ we just need to slide beads down of a total number of $\ell-s$ spaces.} This implies that in each component $m^{(j)}$ of $\bm m$ we need to slide beads down of at most $\ell-s$ spaces. Since $d(\bm m) \geq \ell -s$, similarly to the proof of \cite[Lemma 2.1]{Scopes91} we can conclude that each component $\lambda^{(j)}$ of $\bm \lambda$ has no configuration $ 
        \begin{matrix}

        \bd & \nb \\

        \end{matrix}
$ in runners $i-1$ and $i$.
\end{proof}

{Recall the definition of length of a permutation $\sigma \in \mathfrak{S}_{n}$ as 
$$\ell(\sigma)= |\{(i,j)\ | \ 1 \leq i<j\leq n, \sigma(i) > \sigma(j)\}|.$$ Set $\mathfrak{S}_{n}^k:=\{\sigma \in \mathfrak{S}_n \ | \ \ell(\sigma)=k\}$. Note that $\ell_n^{\max} := \max\{\ell(\sigma)\ | \ \sigma \in \mathfrak{S}_n\} = \frac{n(n-1)}{2}$ and $n!=|\mathfrak{S}_{n}| = \sum\limits_{k=0}^{\ell_n^{\max}}|\mathfrak{S}_{n}^k|$.

\begin{prop}\label{lem2.1}
Fix $i \in I$. Let $B$ be a block of $\mathcal{H}_{r,n}$ such that \eqref{condition_anye} holds and $\delta=\delta_i(B)\geq 0$, and let $\Phi_i(B)$ be the block described in Proposition \ref{bij_mpart}. Set $\ell = \ell^{\max}_{\delta_i(B)}$. Suppose that $\bm{\lambda}$ is an $r$-multipartition of $n$ such that $S^{\bm{\lambda}}$ belongs to the block $B$.
{Then the multipartition $\Phi_i(\bm{\lambda})$ of $n-\delta$ is such that $S^{\Phi_i(\bm{\lambda})}$ belongs to $\Phi_i(B)$ and
\begin{align*}
S^{\bm{\lambda}} \downarrow_{\Phi_i(B)} & \sim \sum_{k=0}^{\ell} |\mathfrak{S}_{\delta}^k| S^{\Phi_i(\bm{\lambda})}\langle\ell-2k\rangle,\\ S^{\Phi_i(\bm{\lambda})} \uparrow^B & \sim \sum_{k=0}^{\ell} |\mathfrak{S}_{\delta}^k| S^{\bm{\lambda}}\langle\ell-2k\rangle.
\end{align*}}
\end{prop}

\begin{proof}
Suppose that $\bm{\mu}$ is a multipartition of $n-\delta$ and $S^{\bm{\mu}}\langle d\rangle$ for some $d \in \Z$ is a factor of $S^{\bm{\lambda}} \downarrow_{\mathcal{H}_{r,n-\delta}}$ using the graded Specht filtration given in Theorem \ref{gradedbranching}. The diagram [$\bm\mu$] can be obtained from [$\bm\lambda$] by removing $\delta$ nodes. In fact, the abacus of $\bm\mu$ is obtained from that of $\bm\lambda$ by successively moving $\delta$ beads one place to the left. Hence, the module $S^{\bm\mu}$ belongs to $\Phi_i(B)$ if and only if the overall effect of $\Phi_i$ is moving $\delta$ beads from each runner $i$ to each runner $i-1$ by point (2) of Proposition \ref{hubblock}.

By Lemma \ref{noconfo-}, in each component of the multipartition $\bm\lambda$, we have no abacus configuration of the type $ 
        \begin{matrix}

        \bd & \nb \\

        \end{matrix}
$ in runners $i-1$ and $i$. This implies that $\bm\lambda$ has no addable $i$-nodes and so $\Phi_i$ consists only of removing $i$-nodes.  Hence, $S^{\bm\mu}$ belongs to $\Phi_i(B)$. The number of ways in which the node removal can be effected is $\delta!$ because they have all the same residue $i$ and so we can remove these nodes in any order. This also shows that there is exactly one $\bm\mu$ that can be found by removing $\delta$ $i$-nodes.

Each order in which we remove the $i$-nodes will determine a shift $d$. Suppose that $A_{\sigma(1)}, \dots,  A_{\sigma(\delta)}$ is the order in which we remove the removable $i$-nodes from $\bm\la$ with $\sigma\in \mathfrak{S}_{\delta}$ of length $k\geq0$. Then
$$d= \ell -2k.$$

We prove this by induction on the length $k$ of the permutation $\sigma$. Let $A_{\sigma(1)}, \dots,  A_{\sigma(\delta)}$ be a sequence of all the removable $i$-nodes of $[\bm\lambda]$. For each $m \in\{1, \ldots, \delta\}$, let $\bm\lambda^{-\sigma(m)}$ be the multipartition of $n~-~m$ obtained by removing the first $m$ nodes in the sequence $A_{\sigma(1)}, \dots,  A_{\sigma(\delta)}$ respecting the order, that is  $[\bm\lambda^{-\sigma(m)}]~=~([\bm\lambda] ~\setminus~ \{A_{\sigma(1)}\}) \setminus \{A_{\sigma(2)}\} \dots \setminus \{ A_{\sigma(m)}\}$. Suppose that the order $A_1 < \dots < A_{\delta}$ on removable $i$-nodes of $\bm\la$ is the one corresponding to the permutation $\sigma=\mathrm{id}$.

If $k=0$, then $\sigma=\mathrm{id}$ and we remove the $i$-nodes in the following order $A_1 ~<~ \dots < ~A_{\delta}$ then
\begin{align*}
    d &= N_{A_1}(\bm\la)+ N_{A_2}(\bm\la^{-1})+ \dots +  N_{A_{\delta}}(\bm\la^{-(\delta-1)})\\
      &= (0 - 0) + (1 - 0) + \dots + (\delta-1 -0)\\
      &= \dfrac{\delta(\delta-1)}{2}\\
      &= \ell.
\end{align*}

Suppose that $k\geq1$ and let $A_{\sigma(1)}, \dots,  A_{\sigma(\delta)}$ be the order in which we remove the $i$-nodes from $\bm\la$. Write $\sigma = s_y\tau$ with $s_y=(y,y+1)$, $\ell(\tau)=k-1$, and $\tau(y)<\tau(y+1)$. Then by inductive hypothesis we know that the shift $d_{\tau}$ corresponding to removing the $i$-nodes in the order $A_{\tau(1)}, \dots,  A_{\tau(\delta)}$ is given by
$$d_{\tau}= \ell -2(k-1).$$
Note that the order $A_{\sigma(1)}, \dots,  A_{\sigma(\delta)}$ corresponding to $\sigma$ is equal to
$$A_{\tau(1)}, \dots,A_{\tau(y+1)},A_{\tau(y)},\dots,  A_{\tau(\delta)}.$$
So, we are removing the $i$-nodes in the same order corresponding to $\tau$ except for the nodes $A_{\tau(y)}, A_{\tau(y+1)}$ that are swapped. Thus, for all $m \neq y, y+1$ we have that $N_{A_{\sigma(m)}}(\bm\la^{-\sigma(m-1)})= N_{A_{\tau(m)}}(\bm\la^{-\tau(m-1)})$. Then it remains to check which relation holds between $N_{A_{\sigma(m)}}(\bm\la^{-\sigma(m-1)})$ and $N_{A_{\tau(m)}}(\bm\la^{-\tau(m-1)})$ for $m\in\{y,y+1\}$.

In terms of abacus configuration, the steps $y$ and $y+1$ of sequences of removable $i$-nodes given by the permutations $\tau$ and $\sigma$ are as shown below (since $\tau(y)<\tau(y+1)$ then $A_{\tau(y)}<A_{\tau(y+1)}$, and so we draw the bead corresponding to $A_{\tau(y)}$ in a lower position with respect to the one of $A_{\tau(y+1)}$). In order to make the visualisation of the abacus displays easier, we can assume that the beads corresponding to $A_{\tau(m)}$ and $A_{\sigma(m)}$ for $m\in\{y,y+1\}$ are in the same component $\la^{(j)}$ of $\bm\la$. %Notice that this assumption is not necessary for the scope of the proof, but it simplifies the representation of the abacus configurations.

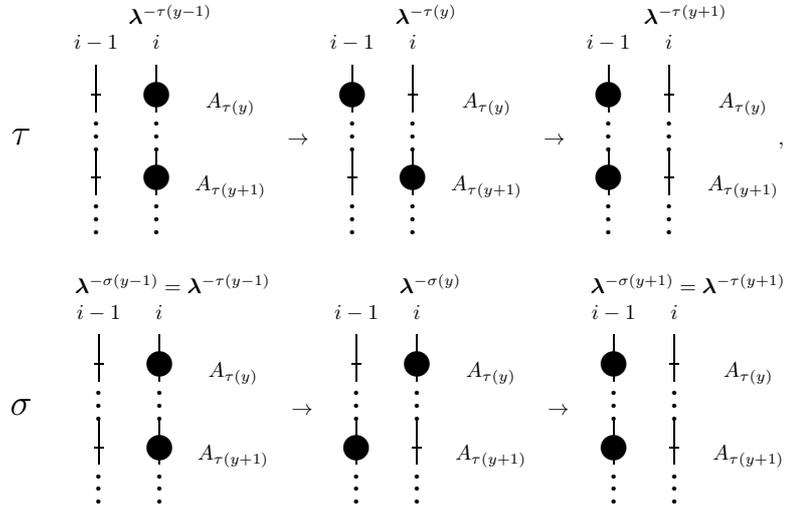
\begin{center}
\scalebox{0.80}{
\begin{tabular}{ccccccc}
&&$\bm\la^{-\tau(y-1)}$ & & $\bm\la^{-\tau(y)}$ & & $\bm\la^{-\tau(y+1)}$\\
{\huge$\tau$} &&$ \begin{NiceMatrix}
    i-1 & i \\
    &&\\    
    \bignb & \bigbd & A_{\tau(y)} \\

    \bigvd & \bigvd & \\
    \bignb & \bigbd & A_{\tau(y+1)}\\
    \bigvd & \bigvd &\\
   \end{NiceMatrix}$
   &
   $\rightarrow$
   &
$ \begin{NiceMatrix}
    i-1 & i \\
    &&\\    
    \bigbd & \bignb & A_{\tau(y)} \\

    \bigvd & \bigvd & \\
    \bignb & \bigbd & A_{\tau(y+1)}\\
    \bigvd & \bigvd &\\
   \end{NiceMatrix}$
   &
   $\rightarrow$
   &
   $ \begin{NiceMatrix}
    i-1 & i \\
    &&\\    
    \bigbd & \bignb & A_{\tau(y)} \\

    \bigvd & \bigvd & \\
    \bigbd & \bignb & A_{\tau(y+1)}\\
    \bigvd & \bigvd &\\
   \end{NiceMatrix},$  
\end{tabular}}
\end{center}

\begin{center}
\scalebox{0.80}{
\begin{tabular}{ccccccc}
& &$\bm\la^{-\sigma(y-1)} =\bm\la^{-\tau(y-1)}$ & & $\bm\la^{-\sigma(y)}$ & & $\bm\la^{-\sigma(y+1)}=\bm\la^{-\tau(y+1)}$\\
{\huge$\sigma$}& &$ \begin{NiceMatrix}
    i-1 & i \\
    &&\\    
    \bignb & \bigbd & A_{\tau(y)} \\

    \bigvd & \bigvd & \\
    \bignb & \bigbd & A_{\tau(y+1)}\\
    \bigvd & \bigvd &\\
   \end{NiceMatrix}$
   &
   $\rightarrow$
   &
$ \begin{NiceMatrix}
    i-1 & i \\
    &&\\    
    \bignb & \bigbd & A_{\tau(y)} \\

    \bigvd & \bigvd & \\
    \bigbd & \bignb & A_{\tau(y+1)}\\
    \bigvd & \bigvd &\\
   \end{NiceMatrix}$
   &
   $\rightarrow$
   &
   $ \begin{NiceMatrix}
    i-1 & i \\
    &&\\    
    \bigbd & \bignb & A_{\tau(y)} \\

    \bigvd & \bigvd & \\
    \bigbd & \bignb & A_{\tau(y+1)}\\
    \bigvd & \bigvd &\\
   \end{NiceMatrix}$  
\end{tabular}}
\end{center}
Notice that $\bm\la^{-\tau(y-1)}$ and $\bm\la^{-\sigma(y)}$ have the same configuration of nodes below $A_{\tau(y)}=A_{\sigma(y+1)}$ and therefore 
\begin{equation*}
   N_{A_{\sigma(y+1)}}(\bm\la^{-\sigma(y)})=N_{A_{\tau(y)}}(\bm\la^{-\tau(y-1)}).
\end{equation*}
Also, we claim that 
\begin{equation*}
    N_{A_{\sigma(y)}}(\bm\la^{-\sigma(y-1)})=N_{A_{\tau(y+1)}}(\bm\la^{-\tau(y)})-2.
\end{equation*}
Let 
\begin{equation*}
    \begin{split}
        \mathrm{add}_{\tau}&=\{\text{addable $i$-nodes of  $\bm\la^{-\tau(y)}$ below $A_{\tau(y+1)}$}\}, \text{ and,}\\
         \mathrm{rem}_{\tau}&=\{\text{removable $i$-nodes of $\bm\la^{-\tau(y)}$ below $A_{\tau(y+1)}$}\}.
    \end{split}
\end{equation*}
Since $N_{A_{\sigma(y)}}(\bm\la^{-\sigma(y-1)})=N_{A_{\tau(y+1)}}(\bm\la^{-\tau(y-1)})$, let 
\begin{equation*}
    \begin{split}
        \mathrm{add}_{\sigma}&=\{\text{addable $i$-nodes of $\bm\la^{-\tau(y-1)}$ below $A_{\tau(y+1)}$}\}, \text{ and,}\\
        \mathrm{rem}_{\sigma} &=\{\text{removable $i$-nodes of $\bm\la^{-\tau(y-1)}$ below $A_{\tau(y+1)}$}\}.\\
    \end{split}
\end{equation*}
The abacus configurations above make it clear that the following equalities hold.
\begin{equation*}
    \begin{split}
    \mathrm{add}_{\sigma}&=\mathrm{add}_{\tau}\setminus \{A_{\tau(y)}\}\\
    \mathrm{rem}_{\sigma}&=\mathrm{rem}_{\tau}\cup \{A_{\tau(y)}\}
    \end{split}
\end{equation*}
As a consequence we get that 
\begin{equation*}
\begin{split}
 N_{A_{\sigma(y)}}(\bm\la^{-\sigma(y-1)})&= |\mathrm{add}_{\sigma}|-|\mathrm{rem}_{\sigma}|\\
 &=|\mathrm{add}_{\tau}|-1-(|\mathrm{rem}_{\tau}|+1)\\
 &=N_{A_{\tau(y+1)}}(\bm\la^{-\tau(y)})-2
\end{split}
\end{equation*}
Hence, we showed that 
\begin{equation*}
\begin{split}
N_{A_{\sigma(y+1)}}(\bm\la^{-\sigma(y)})&=N_{A_{\tau(y)}}(\bm\la^{-\tau(y-1)}), \text{ and}\\
N_{A_{\sigma(y)}}(\bm\la^{-\sigma(y-1)})&=N_{A_{\tau(y+1)}}(\bm\la^{-\tau(y)})-2.\\
\end{split}
\end{equation*}
So, the shift $d_{\sigma}=d_{\tau} -2 =\ell -2k$ as required.

The multiplicity of $S^{\bm\mu}\langle d\rangle$ for each shift $d$ as a factor is the number of permutations of $\mathfrak{S}_{\delta}$ of the same length. This gives the first result about restriction.

Similarly, the second result about induction can be proved by adding $i$-nodes instead of removing $i$-nodes. 
\end{proof}}

\begin{lem}\label{lexor_preserved}
Let $i\in I$. If condition \eqref{condition_anye} holds, then $\Phi_i$ preserves the lexicographic order of multipartitions.
\end{lem}

\begin{proof}
Let $\bm\lambda$ and $\bm\mu$ be multipartitions whose corresponding Specht modules belong to block $B$. Let $\Lambda=(\Lambda_1, \ldots, \Lambda_r)$ and $M=(M_1, \ldots, M_r)$ be the associated sets of $\beta$-numbers,  where $\Lambda_j$ is the $\beta$-set corresponding to $\lambda^{(j)}$ and $M_j$ is the $\beta$-set corresponding to $\mu^{(j)}$ for all $j$.

Let $\bar{\bm\lambda} = \Phi_i(\bm\lambda)$ and $\bar{\bm\mu} = \Phi_i(\bm\mu)$, and let their corresponding sets of $\beta$-numbers be $\bar{\Lambda}=(\bar{\Lambda}_1, \ldots, \bar{\Lambda}_r)$ and $\bar{M}=(\bar{M}_1, \ldots, \bar{M}_r)$. 

Now if $\bm\lambda > \bm\mu$, then as sets of numbers using the lexicographic order $\Lambda>M$. Similarly $\Lambda> M$ implies that $\bm\lambda > \bm\mu$ and we obtain 
$$\bm\lambda > \bm\mu \Leftrightarrow \Lambda> M \Leftrightarrow \Lambda \setminus (\Lambda \cap M)> M\setminus (\Lambda \cap M).$$

Assume $\bm\lambda > \bm\mu$. Let $j_0$ be the minimal $j \in \{1, \ldots, r\}$ such that
$$\Lambda_j \setminus (\Lambda_j \cap M_j) > M_j \setminus (\Lambda_j \cap M_j).$$

Let $\Lambda_{j_0} \setminus (\Lambda_{j_0} \cap M_{j_0}) = \{x_1,\ldots, x_{t}\}$, with $x_l>x_{l+1}$ for $l=1, \ldots,t - 1$, and 
let $M_{j_0} \setminus(\Lambda_{j_0} \cap M_{j_0}) = \{y_1,\ldots,y_{s}\}$, with $y_l>y_{l+1}$ for $l=1, \ldots, s - 1$. Then,
$$\bar{\Lambda}_{j_0} \setminus (\bar{\Lambda}_{j_0} \cap \bar{M}_{j_0}) = \{\phi_i(x_1),\ldots, \phi_i(x_{t})\}$$ 
and
$$\bar{M}_{j_0} \setminus(\bar{\Lambda}_{j_0} \cap \bar{M}_{j_0}) = \{\phi_i(y_1),\ldots, \phi_i(y_{s})\}.$$
Since $\bm\lambda > \bm\mu$, it follows that $x_1 > y_1$. We have three cases to consider.
\begin{description}
\item[Case 1] \textit{$x_1$ belongs to runner $i-1$.} In this case $\phi_i(x_1)=x_1 + 1 > y_l + 1 \geq \phi_i(y_l)$ for all $l$. Hence $\bar{\Lambda} \setminus (\bar{\Lambda} \cap \bar{M})> \bar{M}\setminus (\bar{\Lambda} \cap \bar{M})$, so $\bar{\bm\lambda} > \bar{\bm\mu}$.

\item[Case 2] \textit{$x_1$ belongs to runner $i$.} In this case $\phi_i(x_1)=x_1 - 1$. Then, we want to show that $\phi_i(y_l) \leq x_1-2$ for all $l$. Since $x_1>y_1$, we consider the following three cases.

If $y_1\leq x_1-3$, then $y_l \leq x_1-3$ for all $l$ and $\phi_i(y_l)\leq y_l+1 \leq x_1-2$ for all $l$.

If $y_1 = x_1-2$, then $y_l \leq x_1-2$ for all $l$. If $y_l$ lies in runner $k$ where $k\neq i-1$, then $\phi_i(y_l)\leq y_l \leq x_1-2$ for all $l$. If $y_l$ lies in runner $i-1$, then $y_l\leq x_1-1-e$ and so $\phi_i(y_l) = y_l+1 \leq x_1-e\leq x_1-2$ for all $l$.

If $y_1=x_1-1$, then the abacus of $\mu^{(j_0)}$ presents a configuration $ 
        \begin{matrix}

        \bd & \nb \\

        \end{matrix}
$ in runners $i-1$ and $i$ where the bead corresponds to $y_1$. %This follows by construction of $x_1$ that is the biggest different $\beta$-number between $\lambda^{j_0}$ and $\mu^{j_0}$.
However, by Lemma \ref{noconfo-} a multipartition in the block $B$ cannot have this configuration in runners $i-1$ and $i$. Hence $y_1\neq x_1-1$.

Thus $\phi_i(x_1)=x_1-1>\phi_i(y_l)$ for all $l$, and $\bar{\Lambda} \setminus (\bar{\Lambda} \cap \bar{M})> \bar{M}\setminus (\bar{\Lambda} \cap \bar{M})$, so $\bar{\bm\lambda} > \bar{\bm\mu}$.

%%% SAME PROOF OF ABOVE BUT MORE COMPLICATED THAN NECESSARY
%Clearly $\phi_i(x_1)=x_1-1 \geq y_l + l-1 >y_l +1 \geq \phi_i(y_l)$ for all $l\geq 3$ as $x_1>y_1>y_2>y_3>\cdots$.

%If $y_2$ does not belong to runner $i - 1$, then $\phi_i(y_2) \leq y_2$ and $x_1 -1 > y_2$ as $x_1>y_1>y_2$. So $\phi_i(x_1)= x_1-1>y_2\geq \phi_i(y_2)$. If $y_2$ belongs to runner $i-1$ then $\phi_i(y_2)=y_2+1$. Since $x_1 -1$ and $y_2$ belongs to runner $i-1$ and $y_2 <x_1 - 1$ as above, so $y_2 \leq x_1-1-e$. Hence $\phi_i(x_1) = x_1 -1> y_2 +1 = \phi_i(y_2)$.
%
%If $y_1$ lies in runner $i$, then $\phi_i(x_1)=x_1-1 >y_1-1=\phi_i(y_1)$. If $y_1$ lies in runner $k$, where $k\neq i, i-1$, then $y_1\leq x_1-2$, so $\phi_i(x_1)=x_1-1>x_1-2\geq y_1 =\phi_i(y_1)$. Suppose $y_1$ lies in runner $i-1$. If $y_1=x_1-1$, then the abacus of $\mu^{(j_0)}$ presents a configuration $ 
%        \begin{matrix}
%
%        \bd & \nb \\
%
%        \end{matrix}
%$ in runners $i-1$ and $i$ where the bead corresponds to $y_1$. %This follows by construction of $x_1$ that is the biggest different $\beta$-number between $\lambda^{j_0}$ and $\mu^{j_0}$.
%However, by Lemma \ref{noconfo-} a multipartition in the block $B$ cannot have this configuration in runners $i-1$ and $i$. Hence $y_1\neq x_1-1$, so $y_1+e<x_1$ and $\phi_i(y_1)=y_1+1<x_1-1=\phi_i(x_1)$. Thus $\phi_i(x_1) >\phi_i(y_l)$ for all $l$, and $\bar{\Lambda} \setminus (\bar{\Lambda} \cap \bar{M})> \bar{M}\setminus (\bar{\Lambda} \cap \bar{M})$, so $\bar{\bm\lambda} > \bar{\bm\mu}$.

\item[Case 3] \textit{$x_1$ does not belong to runner $i$ nor to runner $i - 1$.} By similar arguments we see that $\phi_i(x_1)=x_1>\phi_i(y_l)$ for all $l\geq 2$.

If $y_1$ does not lie in runner $i-1$ then $\phi_i(y_1)=y_1$ or $\phi_i(y_1)=y_1-1$. So $\phi_i(x_1)=x_1 > y_1 \geq \phi_i(y_l)$. If $y_1$ lies in runner $i-1$, then $\phi_i(y_1)=y_1+1$. Since $x_1$ is not in runner $i-1$ or in runner $i$, then $y_1<x_1-1$. So again $\phi_i(x_1)>\phi_i(y_1)$. Thus $\phi_i(x_1) >\phi_i(y_l)$ for all $l$, and $\bar{\Lambda} \setminus (\bar{\Lambda} \cap \bar{M})> \bar{M}\setminus (\bar{\Lambda} \cap \bar{M})$, so $\bar{\bm\lambda} > \bar{\bm\mu}$.
\end{description} 
\end{proof}

Now we note that if condition \eqref{condition_anye} holds, $\Phi_i$ preserves the Kleshchev property. Indeed, we have the following result.

\begin{lem}\cite[Lemma 1.9]{Fay07}\label{Kle_preserved} Suppose $\bm\lambda$ is a multipartition, and that $[\bm\lambda]$ has no addable nodes of residue $i$ for $i\in I$. Then $\bm\lambda$ is a Kleshchev multipartition if and only if $\Phi_i(\bm\lambda)$ is.
\end{lem}

{Given that, we just need to notice that condition \eqref{condition_anye} implies that $[\bm\lambda]$ has no addable nodes of residue $i$. Hence, we can conclude that if condition \eqref{condition_anye} holds, then Kleshchev multipartitions are preserved by $\Phi_i$.

Similarly to Lemma 2.4 in \cite{Scopes91}, we have the following proposition. For this, let $\bm\la, \bm\mu$ be Kleshchev multipartitions of $n$ and define the \textbf{graded Cartan matrix} $\bm C (v)=(c_{\bm\la\bm\mu}(v))$ of $\mathcal{H}_{r,n}$ where $$c_{\bm\la\bm\mu}(v)=\sum\limits_{d\in \Z}[P^{\bm\la}:D^{\bm\mu}\langle d\rangle]v^d$$ is the graded multiplicity of the simple module $D^{\bm\mu}$ in the principal indecomposable module $P^{\bm \lambda}$. Recall that if we denote by $\bm D(v)$ the graded decomposition matrix of $\mathcal{H}_{r,n}$, the graded Brauer-Humphreys reciprocity tells that
\begin{equation}\label{cartan_dec_mtx}
\bm C(v) = \bm D(v)^t \bm D(v)
\end{equation}
where $\bm D(v)^t$ is the transpose of the matrix $\bm D(v)$.

\begin{thm}
Fix $i \in I$. Let $B$ be a block of $\mathcal{H}_{r,n}$ such that \eqref{condition_anye} holds and $\delta=\delta_i(B)\geq 0$. Set $\ell = \ell^{\max}_{\delta_i(B)}$. Suppose that $\bm{\lambda}$ is a $r$-multipartition of $n$ such that $S^{\bm{\lambda}}$ belongs to the block $B$. Then
\begin{enumerate}
\item[1.] $D^{\bm{\lambda}} \downarrow_{\Phi_i(B)}  \sim \sum\limits_{k=0}^{\ell} |\mathfrak{S}_{\delta}^k| D^{\Phi_i(\bm{\lambda})}\langle\ell-2k\rangle$.
\item[2.]  $D^{\Phi_i(\bm{\lambda})} \uparrow^B \sim \sum\limits_{k=0}^{\ell} |\mathfrak{S}_{\delta}^k| D^{\bm{\lambda}}\langle\ell-2k\rangle.$
\item[3.] The blocks $B$ and $\Phi_i(B)$ have the same graded decomposition matrix.
\item[4.] The blocks $B$ and $\Phi_i(B)$ have the same graded Cartan matrix.
\end{enumerate}
\end{thm}

\begin{proof}
Let $\bm\lambda_1>\bm\lambda_2> \ldots >\bm\lambda_y$ be the Kleshchev multipartitions whose Specht modules belong to $B$; then $\Phi_i(\bm\lambda_1)>\Phi_i(\bm\lambda_2)> \ldots >\Phi_i(\bm\lambda_y)$ are the Kleshchev multipartitions whose Specht modules belong to $\Phi_i(B)$ by Lemmas \ref{lexor_preserved} and \ref{Kle_preserved}. Suppose
\begin{equation}\label{dec_numb}
S^{\bm\lambda} \sim\sum_{j=1}^{y} d_{\bm\lambda\bm\lambda_j}(q) D^{\bm\lambda_j}, \quad d_{\bm\lambda\bm\lambda_j}(v)\in \mathbb{N}[v,v^{-1}] 
\end{equation}
%If $\bm\lambda \trianglerighteq \bm\mu$, then $\bm\lambda \geq \bm\mu$, that is equivalent to say if $\bm\lambda < \bm\mu$, then $\bm\lambda \ntrianglerighteq \bm\mu$. Hence, 
By point (2) of Theorem \ref{decmtxHQ} we have that if $\bm\lambda < \bm\mu$, then $d_{\bm\lambda\bm\mu}(v)=0$.

Hence, \eqref{dec_numb} becomes
\begin{equation}\label{dec_numb_values}
S^{\bm\lambda} \sim\sum_{j=1}^{y} d_{\bm\lambda\bm\lambda_j}(v) D^{\bm\lambda_j}, \quad d_{\bm\lambda\bm\lambda_j}(v)=\begin{cases}
1 & \text{if } \bm\lambda=\bm\lambda_j\\
0 & \text{if } \bm\lambda<\bm\lambda_j\\
\end{cases}.
\end{equation}
In particular, $S^{\bm\lambda_y}=D^{\bm\lambda_y}$. We want to prove points 1. and 2. for $\bm\la_1,  \ldots, \bm\la_y$.

By Proposition \ref{lem2.1} we have
\begin{equation*}
    \begin{split}
        S^{\bm{\lambda}_y} \downarrow_{\Phi_i(B)} & \sim \sum_{k=0}^{\ell} |\mathfrak{S}_{\delta}^k| v^{\ell-2k}S^{\Phi_i(\bm{\lambda}_y)}, \text{ and},\\
        S^{\Phi_i(\bm{\lambda}_y)} \uparrow^B & \sim \sum_{k=0}^{\ell} |\mathfrak{S}_{\delta}^k|v^{\ell-2k} S^{\bm{\lambda}_y},
    \end{split}
\end{equation*}
so, since $S^{\bm\lambda_y}=D^{\bm\lambda_y}$ we get the result:
\begin{equation*}
    \begin{split}
        D^{\bm{\lambda}_y} \downarrow_{\Phi_i(B)} & \sim \sum_{k=0}^{\ell} |\mathfrak{S}_{\delta}^k| v^{\ell-2k}D^{\Phi_i(\bm{\lambda}_y)}, \text{ and},\\
        D^{\Phi_i(\bm{\lambda}_y)} \uparrow^B & \sim \sum_{k=0}^{\ell} |\mathfrak{S}_{\delta}^k|v^{\ell-2k} D^{\bm{\lambda}_y}.
    \end{split}
\end{equation*}

Now, suppose that points 1. and 2. holds for $\bm\lambda_l, \ldots, \bm\lambda_{y}$ with $1<l\leq y$, that is for $l\leq j\leq y$ we have
\begin{equation*}
    \begin{split}
        D^{\bm{\lambda}_j} \downarrow_{\Phi_i(B)} & \sim \sum_{k=0}^{\ell} |\mathfrak{S}_{\delta}^k| v^{\ell-2k}D^{\Phi_i(\bm{\lambda}_j)}, \text{ and},\\
        D^{\Phi_i(\bm{\lambda}_j)} \uparrow^B & \sim \sum_{k=0}^{\ell} |\mathfrak{S}_{\delta}^k|v^{\ell-2k} D^{\bm{\lambda}_j}.
    \end{split}
\end{equation*}

Thus, we want to prove points 1. and 2. for $\bm\lambda_{l-1}$. 
Then
\begin{equation*}
\begin{split}
(S^{\bm\lambda_{l-1}})\downarrow_{\Phi_i(B)}\uparrow^{B} &\underset{\mathclap{\tikz \node {$\uparrow$} node [below=1ex] {by Prop. \ref{lem2.1}};}}{\sim} \sum_{k,h=0}^{\ell} |\mathfrak{S}_{\delta_{i}(B)}^k| |\mathfrak{S}_{\delta_{i}(B)}^h| v^{\ell-2(k+h)} S^{\bm\lambda_{l-1}}\\
&\underset{\mathclap{\tikz \node {$\uparrow$} node [below=1ex] {by \eqref{dec_numb_values}};}}{\sim} \sum_{k,h=0}^{\ell} |\mathfrak{S}_{\delta_{i}(B)}^k| |\mathfrak{S}_{\delta_{i}(B)}^h|v^{\ell-2(k+h)} \left(\sum_{j=l}^{y} (d_{\bm\lambda_{l-1}\bm\lambda_j}(v)D^{\bm\lambda_j} + D^{\bm\lambda_{l-1}})\right),
\end{split}
\end{equation*}
and, applying first \eqref{dec_numb_values} and then Proposition \ref{lem2.1} together with the inductive  hypothesis on $j\geq l$
\begin{equation*}
(S^{\bm\lambda_{l-1}})\downarrow_{\Phi_i(B)}\uparrow^{B} \sim \sum_{j=l}^{y} \left(\sum_{k,h=0}^{\ell} |\mathfrak{S}_{\delta_{i}(B)}^k| |\mathfrak{S}_{\delta_{i}(B)}^h|v^{\ell-2(k+h)}\right)d_{\bm\lambda_{l-1}\bm\lambda_j}(v)D^{\bm\lambda_j} + (D^{\bm\lambda_{l-1}})\downarrow_{\Phi_i(B)}\uparrow^{B}.
\end{equation*}
So 
%For the short way to conclude.
%$$D^{\bm\lambda_l})\downarrow_{\Phi_i(B)}\uparrow^{B} \sim (\delta_i(B)!)^2 D^{\bm\lambda_l}.$$
\begin{equation}\label{Dleq1}
D^{\bm\lambda_{l-1}}\downarrow_{\Phi_i(B)}\uparrow^{B} \sim \left(\sum_{k,h=0}^{\ell} |\mathfrak{S}_{\delta_{i}(B)}^k| |\mathfrak{S}_{\delta_{i}(B)}^h|v^{\ell-2(k+h)} \right)D^{\bm\lambda_{l-1}}.
\end{equation}
%Short way to conclude the proof without giving all the details.
%Hence, using Proposition \ref{lem2.1} and  \eqref{dec_numb_values} we can conclude that
%$$D^{\bm\lambda_l}\downarrow_{\Phi_i(B)} \sim \delta_i(B)! D^{\Phi_i(\bm\lambda_l)} \quad\text{ and }\quad D^{\Phi_i(\bm\lambda_l)}\uparrow^{B} \sim \delta_i(B)! D^{\bm\lambda_l}$$
%and
%\begin{equation}\label{point3}
%S^{\Phi_i(\bm\lambda_l)} \sim \sum_{j=1}^{y} d_{\bm\lambda_l\bm\lambda_j} D^{\Phi_i(\bm\lambda_j)}.
%\end{equation}
%Thus, points 1. and 2. are proved by induction, and points 3. and 4. follow immediately from \eqref{point3}.
%Long way to conclude the proof giving all the details.
Now notice that for some $\alpha_j(v), \beta_j(v) \in \mathbb{N}[v,v^{-1}]$
\begin{align}
D^{\bm\lambda_{l-1}}\downarrow_{\Phi_i(B)} &\sim \sum_{j=l-1}^{y} \alpha_j(v) D^{\Phi_i(\bm\lambda_j)},\label{alpharel}\\
D^{\Phi_i(\bm\lambda_{l-1})}\uparrow^{B} &\sim \sum_{j=l-1}^{y} \beta_j(v) D^{\bm\lambda_j}.\label{betarel}
\end{align}
Then, by the hypothesis on $j\geq l$ and using \eqref{alpharel}, \eqref{betarel}
\begin{equation}\label{Dleq2}
\begin{split}
(D^{\bm\lambda_{l-1}})\downarrow_{\Phi_i(B)}\uparrow^{B} \sim &\sum_{j=l}^{y} \left(\left(\sum_{k=0}^{\ell} |\mathfrak{S}_{\delta}^k|v^{\ell-2k}\right)\alpha_j(v) + \alpha_{l-1}(v)\beta_j(v)\right)D^{\bm\lambda_j} \\
&+ \alpha_{l-1}(v)\beta_{l-1}(v) D^{\bm\lambda_{l-1}}.    
\end{split}
\end{equation}
Now combining \eqref{Dleq1} and \eqref{Dleq2}, we get
\begin{equation*}
\begin{split}
     &\left(\sum_{k,h=0}^{\ell} |\mathfrak{S}_{\delta_{i}(B)}^k| |\mathfrak{S}_{\delta_{i}(B)}^h|v^{\ell-2(k+h)} \right)D^{\bm\lambda_{l-1}} \\
     &\sim \sum_{j=l}^{y} \left(\left(\sum_{k=0}^{\ell} |\mathfrak{S}_{\delta}^k|v^{\ell-2k}\right)\alpha_j(v) + \alpha_{l-1}(v)\beta_j(v)\right)D^{\bm\lambda_j} + \alpha_{l-1}(v)\beta_{l-1}(v) D^{\bm\lambda_{l-1}}
\end{split}
\end{equation*}
and so by the uniqueness of the composition series of $(D^{\bm\lambda_{l-1}})\downarrow_{\Phi_i(B)}\uparrow^{B}$ we have 
\begin{equation}\label{alphabeta}
  \alpha_{l-1}(v)\beta_{l-1}(v)=\sum_{k,h=0}^{\ell} |\mathfrak{S}_{\delta_{i}(B)}^k| |\mathfrak{S}_{\delta_{i}(B)}^h|v^{\ell-2(k+h)}  
\end{equation}
and $$\alpha_j(v) = 0 = \beta_j(v) \text{ for all }l \leq j\leq y.$$
Thus, by \eqref{alpharel} and \eqref{betarel} we obtain
$$D^{\bm\lambda_{l-1}}\downarrow_{\Phi_i(B)} \sim \alpha_{l-1}(v) D^{\Phi_i(\bm\lambda_{l-1})} \quad\text{ and }\quad D^{\Phi_i(\bm\lambda_{l-1})}\uparrow^{B} \sim \beta_{l-1}(v) D^{\bm\lambda_{l-1}}$$ 
with $\alpha_{l-1}(v)$ and $\beta_{l-1}(v)$ satisfying \eqref{alphabeta}.
Hence, using Proposition \ref{lem2.1} and  \eqref{dec_numb_values} we have that
\begin{equation*}
(S^{\bm\lambda_{l-1}})\downarrow_{\Phi_i(B)} \sim \sum_{j=l}^{y} \left(\sum_{k=0}^{\ell} |\mathfrak{S}_{\delta_{i}(B)}^k| v^{\ell-2k} \right)d_{\bm\lambda_{l-1}\bm\lambda_j}(v)D^{\Phi_i(\bm\lambda_j)} + \alpha_{l-1}(v) D^{\Phi_i(\bm\lambda_{l-1})}
\end{equation*}
and,
\begin{equation*}
\begin{split}
   &\left(\sum_{k=0}^{\ell} |\mathfrak{S}_{\delta_{i}(B)}^k| v^{\ell-2k} \right) S^{\Phi_i(\bm\lambda_{l-1})}\\
   &\sim \left(\sum_{k=0}^{\ell} |\mathfrak{S}_{\delta_{i}(B)}^k| v^{\ell-2k} \right) \left(\sum_{j=l}^{y} d_{\Phi_i(\bm\lambda_{l-1})\Phi_i(\bm\lambda_j)}(v)D^{\Phi_i(\bm\lambda_j)} + D^{\Phi_i(\bm\lambda_{l-1})}\right) 
\end{split}
\end{equation*}
and so we can conclude that $d_{\bm\lambda_{l-1}\bm\lambda_j}(v)=d_{\Phi_i(\bm\lambda_{l-1})\Phi_i(\bm\lambda_j)}(v)$ for all $l \leq j \leq y$, $\alpha_{l-1}(v)=\left(\sum_{k=0}^{\ell} |\mathfrak{S}_{\delta_{i}(B)}^k| v^{\ell-2k} \right)$ and so $\beta_{l-1}(v)=\left(\sum_{k=0}^{\ell} |\mathfrak{S}_{\delta_{i}(B)}^k| v^{\ell-2k} \right)$.
Therefore,
$$D^{\bm\lambda_{l-1}}\downarrow_{\Phi_i(B)} \sim \left(\sum_{k=0}^{\ell} |\mathfrak{S}_{\delta_{i}(B)}^k| v^{\ell-2k} \right) D^{\Phi_i(\bm\lambda_{l-1})}$$
and $$D^{\Phi_i(\bm\lambda_{l-1})}\uparrow^{B} \sim \left(\sum_{k=0}^{\ell} |\mathfrak{S}_{\delta_{i}(B)}^k| v^{\ell-2k} \right) D^{\bm\lambda_{l-1}}$$
and
\begin{equation}\label{point3}
S^{\Phi_i(\bm\lambda_{l-1})} \sim \sum_{j=l-1}^{y} d_{\bm\lambda_{l-1}\bm\lambda_j}(v) D^{\Phi_i(\bm\lambda_j)}.
\end{equation}
Thus, points 1. and 2. are proved for any $\bm\la_j$ with $1\leq j \leq y$, point 3. follows immediately from \eqref{point3}, and point 4. follows from \eqref{cartan_dec_mtx}.
\end{proof}}

\section*{Acknowledgements}
This paper forms part of work towards a PhD degree under the supervision of Dr. Sin\'ead Lyle at University of East Anglia, and the author wishes to thank her for her direction and support throughout.

\bibliographystyle{amsalpha}

%\listoftodos

\end{document}